\documentclass[12pt]{amsart}
\usepackage{amsmath, amsthm, amssymb}
\usepackage[margin=2cm]{geometry}
\usepackage{epsfig}
\usepackage{rawfonts}
\usepackage{enumerate}
\usepackage{graphics}
\usepackage{multirow}
\usepackage{xspace}
\usepackage{graphicx}

\usepackage{pgf,tikz,pgfplots}
\usepackage{mathrsfs}
\usetikzlibrary{arrows}
\usepackage{amsmath}
\usepackage{amsfonts}
\usepackage{amssymb}
\usepackage{amsthm}
\usepackage{graphicx}
\usepackage{booktabs}
\usepackage{caption}
\usepackage{listings}
\usepackage{setspace}
\usepackage[mathscr]{eucal}
\usepackage{pgfplots}
\usepackage{hyperref}
\usepackage{wrapfig}
\usepackage{floatflt,epsfig}
\usepackage{ dsfont }
\usepackage{amscd}
\usepackage{tikz-cd}
\usepackage{fancyhdr}
\usepackage[all]{xy}
\usepackage{latexsym}
\usepackage{amscd}
\usepackage{pifont}
\usepackage{eufrak}
\usepackage{subfig}
\usepackage{easyReview}

\newcommand{\numberset}{\mathbb}

\newcommand{\Z}{\numberset{Z}}

\def\ZZ{{\mathbb Z}}

\newcommand{\lt}{\mathop{\rm in}\nolimits}

\newcommand{\ord}{<^{\mathsf{P}}_{\mathrm{lex}}}

\newcommand{\cC}{\mathcal{C}}

\newcommand{\cP}{\mathcal{P}}
\newcommand{\cH}{\mathcal{H}}

\newcommand{\cB}{\mathcal{B}}

\newcommand{\cW}{\mathcal{W}}

\newcommand{\cS}{\mathcal{S}}

\newcommand{\cG}{\mathcal{G}}

\theoremstyle{plain}

\theoremstyle{theorem}

\newtheorem{defn}{Definition}[section]

\newtheorem{thm}[defn]{Theorem}
\newtheorem{lemma}[defn]{Lemma}

\newtheorem{rmk}[defn]{Remark}

\newtheorem{alg}[defn]{Algorithm}
\theoremstyle{remark}

\makeatletter
\@namedef{subjclassname@2020}{\textup{2020} Mathematics Subject Classification}
\makeatother

\begin{document}
		
		\title[ON GR\"OBNER BASES AND COHEN-MACAULAY PROPERTY OF CLOSED PATH POLYOMINOES]{ON GR\"OBNER BASES AND COHEN-MACAULAY PROPERTY OF CLOSED PATH POLYOMINOES}
		
		\author{CARMELO CISTO}
		\address{Universit\'{a} di Messina, Dipartimento di Scienze Matematiche e Informatiche, Scienze Fisiche e Scienze della Terra\\
			Viale Ferdinando Stagno D'Alcontres 31\\
			98166 Messina, Italy}
		\email{carmelo.cisto@unime.it}

		\author{FRANCESCO NAVARRA}
		\address{Universit\'{a} di Messina, Dipartimento di Scienze Matematiche e Informatiche, Scienze Fisiche e Scienze della Terra\\
			Viale Ferdinando Stagno D'Alcontres 31\\
			98166 Messina, Italy}
		\email{francesco.navarra@unime.it}
		
		\author{ROSANNA UTANO}
		\address{Universit\'{a} di Messina, Dipartimento di Scienze Matematiche e Informatiche, Scienze Fisiche e Scienze della Terra\\
			Viale Ferdinando Stagno D'Alcontres 31\\
			98166 Messina, Italy}
		\email{rosanna.utano@unime.it}

		

		\dedicatory{Dedicated to the memory of Gaetana Restuccia}
		
		\begin{abstract}
 In this paper we introduce some monomial orders for the class of closed path polyominoes and we prove that the set of the generators of the polyomino ideal attached to a closed path forms the reduced Gr\"obner basis with respect to these monomial orders. It is known that the polyomino ideal attached to a closed path containing an L-configuration or a ladder of at least three steps, equivalently having no zig-zag walks, is prime. As a consequence, we obtain that the coordinate ring of a closed path having no zig-zag walks is a normal Cohen-Macaulay domain.  	
		\end{abstract}

		\maketitle
		
	\section{Introduction}

	\noindent 
	Let $X=(x_{ij})$ be an $m \times n$ matrix of indeterminates. An interesting topic in Commutative Algebra is the studying the ideal of the $t$-minors of $X$ for any integer 
	$1\leq t\leq \min\{m,n\}$. Many matematicians investigated the main algebraic properties of such ideals, called determinantal ideals. See for example \cite{conca1}, \cite{conca2}, \cite{conca3}, \cite{ladder deter varieties} about the ideals
	generated by all $t$-minors of a one or two sided ladder, \cite{adiajent1}, \cite{adjent 2}, \cite{adiajent3} about the ideals of adjacent 2-minors, \cite{2.n} about the ideals generated by an arbitrary set
	of $2$-minors in a $2\times n$ matrix, or also \cite{Ene-qureshi} about the ideals generated by 2-minors associated to a graph. For further references see \cite{Bruns Vetter}.\\	
	\noindent In 1953 S.W. Golomb coined the term \textit{polyomino} to indicate a finite collection of unitary squares joined edge by edge (\cite{golomb}). These polygons have been studied in Combinatorial Mathematics, in particular in some tiling problems of the plane. In 2012 they have been linked to Commutative Algebra by A.A. Qureshi (\cite{Qureshi}): if $\cP$ is a polyomino, $K$ is a field and $S$ is the polynomial ring over $K$ in the variables $x_a$ with $a\in V(\cP)$, set of the vertices of $\cP$, then we can associate to $\cP$ the ideal generated by all inner 2-minors of $\cP$. This ideal is called the \textit{polyomino ideal} of $\cP$ and it is denoted by $I_{\cP}$.
	\noindent Many mathematicians have taken an interest in classifying all polyominoes, for which the quotient ring $K[\cP]=S/I_{\cP}$ is a normal Cohen-Macaulay domain. \\
	The primality of $I_{\cP}$ is studied in several papers, see \cite{Cisto_Navarra}, \cite{Cisto_Navarra2}, \cite{Simple equivalent balanced}, \cite{def balanced}, \cite{Not simple with localization}, \cite{Trento}, \cite{Trento2}, \cite{Simple are prime}, \cite{Shikama}, \cite{Shikam rettangolo meno }. Moreover in \cite{Simple equivalent balanced} and \cite{Simple are prime} the authors prove that if $\cP$ is a simple polyomino then $K[\cP]$ is a normal Cohen-Macaulay domain. In other papers some new classes of non-simple polyominoes are examined. In \cite{Not simple with localization} and \cite{Shikama}, the authors show that the polyominoes obtained by removing a convex polyomino from a rectangle are prime, generalizing the same result for the rectangular polyominoes minus an internal rectangle proved in \cite{Shikama}. In \cite{Trento} the authors introduce a particular sequence of inner intervals of $\cP$, called a zig-zag walk, and they prove that $\cP$ does not contain zig-zag walks if $I_{\cP}$ is prime. It seems that the non-existence of zig-zag walks in a polyomino could characterize its primality [Conjecture 4.6, \cite{Trento}]. In \cite{Cisto_Navarra} the authors support this conjecture introducing a new class of polyominoes, called closed paths, and showing that having no zig-zag walks is a necessary and sufficient condition for their primality. An analogous result is proved in \cite{Cisto_Navarra2} for the weakly closed path polyominoes. In \cite{Trento2} the authors study the reduced Gr\"obner basis of polyomino ideals and introduce some conditions in order to the generators of $I_{\cP}$ form the reduced Gr\"obner basis with respect to some suitable degree reverse lexicographic monomial orders. Eventually, for further references about several algebraic properties of polyomino ideals we report \cite{Andrei}, \cite{Herzog rioluzioni lineari}, \cite{L-convessi} and \cite{Trento3}.\\
	\noindent In this paper we study the Gr\"obner bases of the polyomino ideal attached to a closed path and we show that there exist some monomial orders such that the set of generators of the ideal forms the reduced Gr\"obner basis with respect to these orders. In Section \ref*{Section: Introduction} we introduce the notations about polyominoes and closed paths. In Section \ref*{Section: Properties} we provide a class of monomial orderings that generalizes the class introduced in \cite{Qureshi} by Qureshi and we give some conditions on such orderings in order to the $S$-polynomial of two generators of $I_{\cP}$ attached to a collection of cells reduces to $0$ modulo the set of generators of $I_{\cP}$. In Section \ref{Section: Grobner basis of polyomino ideal of closed path} we introduce some new configurations in a closed path, in particular the $W$-pentominoes, the LD-horizontal and vertical skew tetrominoes and hexominoes and the RW-heptominoes. For each case in which one of the previous configurations is not in the closed path we provide a set of suitable vertices which allow us to define some particular suitable monomial orders. Moreover in Definition \ref{def:Yij} we provide a pseudo-algorithm to deal the general case. Finally, we prove that the set of generators of the polyomino ideal attached to a closed path is the reduced Gr\"obner basis with respect to a suitable choice of the monomial order, so the coordinate ring of a closed path having no zig-zag walks is a normal Cohen-Macaulay domain. We conclude the paper giving an example of a non-simple polyomino whose universal Gr\"obner basis is not square-free and providing some related open questions.   
	\section{Polyominoes, closed paths and polyomino ideals}\label{Section: Introduction}
\noindent Let $(i,j),(k,l)\in \Z^2$. We say that $(i,j)\leq(k,l)$ if $i\leq k$ and $j\leq l$. Consider $a=(i,j)$ and $b=(k,l)$ in $\Z^2$ with $a\leq b$. The set $[a,b]=\{(m,n)\in \Z^2: i\leq m\leq k,\ j\leq n\leq l \}$ is called an \textit{interval} of $\Z^2$. 
In addiction, if $i< k$ and $j<l$ then $[a,b]$ is a \textit{proper} interval. In such a case we say $a, b$ the \textit{diagonal corners} of $[a,b]$ and $c=(i,l)$, $d=(k,j)$ the \textit{anti-diagonal corners} of $[a,b]$. If $j=l$ (or $i=k$) then $a$ and $b$ are in \textit{horizontal} (or \textit{vertical}) \textit{position}. We denote by $]a,b[$ the set $\{(m,n)\in \Z^2: i< m< k,\ j< n< l\}$. A proper interval $C=[a,b]$ with $b=a+(1,1)$ is called a \textit{cell} of $\ZZ^2$; moreover, the elements $a$, $b$, $c$ and $d$ are called respectively the \textit{lower left}, \textit{upper right}, \textit{upper left} and \textit{lower right} \textit{corner} of $C$. The sets $\{a,c\}$, $\{c,b\}$, $\{b,d\}$ and $\{a,d\}$ are the \textit{edges} of $C$. We put $V(C)=\{a,b,c,d\}$ and $E(C)=\{\{a,c\},\{c,b\},\{b,d\},\{a,d\}\}$. \\
Let $\cS$ be a non-empty collection of cells in $\Z^2$. The set of the vertices and the edges of $\cS$ are respectively $V(\cS)=\bigcup_{C\in \cS}V(C)$ and $E(\cS)=\bigcup_{C\in \cS}E(C)$. If $C$ and $D$ are two distinct cells of $\cS$, then a \textit{walk} from $C$ to $D$ in $\cS$ is a sequence $\cC:C=C_1,\dots,C_m=D$ of cells of $\ZZ^2$ such that $C_i \cap C_{i+1}$ is an edge of $C_i$ and $C_{i+1}$ for $i=1,\dots,m-1$. In addition, if $C_i \neq C_j$ for all $i\neq j$, then $\cC$ is called a \textit{path} from $C$ to $D$. We say that $C$ and $D$ are connected in $\cS$ if there exists a path of cells in $\cS$ from $C$ to $D$. 
A \textit{polyomino} $\cP$ is a non-empty, finite collection of cells in $\Z^2$ where any two cells of $\cP$ are connected in $\cP$. For instance, see Figure \ref{Figura: Polimino introduzione}.
\begin{figure}[h]
	\centering
	\includegraphics[scale=0.5]{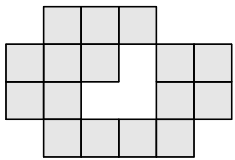}
	\caption{A polyomino.}
	\label{Figura: Polimino introduzione}
\end{figure}

	 \noindent We say that a polyomino $\cP$ is \textit{simple} if for any two cells $C$ and $D$ not in $\cP$ there exists a path of cells not in $\cP$ from $C$ to $D$. A finite collection of cells $\cH$ not in $\cP$ is a \textit{hole} of $\cP$ if any two cells of $\cH$ are connected in $\cH$ and $\cH$ is maximal with respect to set inclusion. For example, the polyomino in Figure \ref{Figura: Polimino introduzione} is not simple with an hole. Obviously, each hole of $\cP$ is a simple polyomino and $\cP$ is simple if and only if it has not any hole.\\
	 Consider two cells $A$ and $B$ of $\Z^2$ with $a=(i,j)$ and $b=(k,l)$ as the lower left corners of $A$ and $B$ and $a\leq b$. A \textit{cell interval} $[A,B]$ is the set of the cells of $\Z^2$ with lower left corner $(r,s)$ such that $i\leqslant r\leqslant k$ and $j\leqslant s\leqslant l$. If $(i,j)$ and $(k,l)$ are in horizontal (or vertical) position, we say that the cells $A$ and $B$ are in \textit{horizontal} (or \textit{vertical}) \textit{position}.\\
	 Let $\cP$ be a polyomino. Consider  two cells $A$ and $B$ of $\cP$ in vertical or horizontal position. 	 
	 The cell interval $[A,B]$, containing $n>1$ cells, is called a
	 \textit{block of $\cP$ of rank n} if all cells of $[A,B]$ belong to $\cP$. The cells $A$ and $B$ are called \textit{extremal} cells of $[A,B]$. Moreover, a block $\cB$ of $\cP$ is \textit{maximal} if there does not exist any block of $\cP$ which contains properly $\cB$. It is clear that an interval of $\ZZ^2$ identifies a cell interval of $\ZZ^2$ and vice versa, hence we can associated to an interval $I$ of $\ZZ^2$ the corresponding cell interval denoted by $\cP_{I}$. A proper interval $[a,b]$ is called an \textit{inner interval} of $\cP$ if all cells of $\cP_{[a,b]}$ belong to $\cP$. An interval $[a,b]$ with $a=(i,j)$, $b=(k,j)$ and $i<k$ is called a \textit{horizontal edge interval} of $\cP$ if the sets $\{(\ell,j),(\ell+1,j)\}$ are edges of cells of $\cP$ for all $\ell=i,\dots,k-1$. In addition, if $\{(i-1,j),(i,j)\}$ and $\{(k,j),(k+1,j)\}$ do not belong to $E(\cP)$, then $[a,b]$ is called a \textit{maximal} horizontal edge interval of $\cP$. We define similarly a \textit{vertical edge interval} and a \textit{maximal} vertical edge interval. \\
	Following \cite{Trento} we recall the definition of a \textit{zig-zag walk} of $\cP$. A zig-zag walk of $\cP$ is a sequence $\cW:I_1,\dots,I_\ell$ of distinct inner intervals of $\cP$ where, for all $i=1,\dots,\ell$, the interval $I_i$ has either diagonal corners $v_i$, $z_i$ and anti-diagonal corners $u_i$, $v_{i+1}$ or anti-diagonal corners $v_i$, $z_i$ and diagonal corners $u_i$, $v_{i+1}$, such that:
	\begin{enumerate}
		\item $I_1\cap I_\ell=\{v_1=v_{\ell+1}\}$ and $I_i\cap I_{i+1}=\{v_{i+1}\}$, for all $i=1,\dots,\ell-1$;
		\item $v_i$ and $v_{i+1}$ are on the same edge interval of $\cP$, for all $i=1,\dots,\ell$;
		\item for all $i,j\in \{1,\dots,\ell\}$ with $i\neq j$, there exists no inner interval $J$ of $\cP$ such that $z_i$, $z_j$ belong to $V(J)$.
	\end{enumerate}

\noindent In according to \cite{Cisto_Navarra}, we recall the definition of a \textit{closed path polyomino}, and the configuration of cells characterizing its primality. We say that a polyomino $\cP$ is a \textit{closed path} if it is a sequence of cells $A_1,\dots,A_n, A_{n+1}$, $n>5$, such that:
	\begin{enumerate}
		\item $A_1=A_{n+1}$;
		\item $A_i\cap A_{i+1}$ is a common edge, for all $i=1,\dots,n$;
		\item $A_i\neq A_j$, for all $i\neq j$ and $i,j\in \{1,\dots,n\}$;
		\item For all $i\in\{1,\dots,n\}$ and for all $j\notin\{i-2,i-1,i,i+1,i+2\}$ then $V(A_i)\cap V(A_j)=\emptyset$, where $A_{-1}=A_{n-1}$, $A_0=A_n$, $A_{n+1}=A_1$ and $A_{n+2}=A_2$. 
	\end{enumerate}
\begin{figure}[h]
 	\centering
 	\includegraphics[scale=0.4]{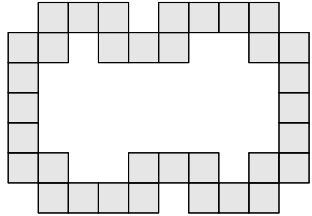}
 	\caption{An example of a closed path.}
\end{figure}
 	A path of five cells $C_1, C_2, C_3, C_4, C_5$ of $\cP$ is called an \textit{L-configuration} if the two sequences $C_1, C_2, C_3$ and $C_3, C_4, C_5$ go in two orthogonal directions. A set $\cB=\{\cB_i\}_{i=1,\dots,n}$ of maximal horizontal (or vertical) blocks of rank at least two, with $V(\cB_i)\cap V(\cB_{i+1})=\{a_i,b_i\}$ and $a_i\neq b_i$ for all $i=1,\dots,n-1$, is called a \textit{ladder of $n$ steps} if $[a_i,b_i]$ is not on the same edge interval of $[a_{i+1},b_{i+1}]$ for all $i=1,\dots,n-2$. For instance, in Figure \ref{Figura:L conf + Ladder} there is a closed path having an L-configuration and a ladder of three steps. We recall that a closed path has no zig-zag walks if and only if it contains an L-configuration or a ladder of at least three steps (see \cite{Cisto_Navarra}).
 \begin{figure}[h]
 	\centering
 	\includegraphics[scale=0.7]{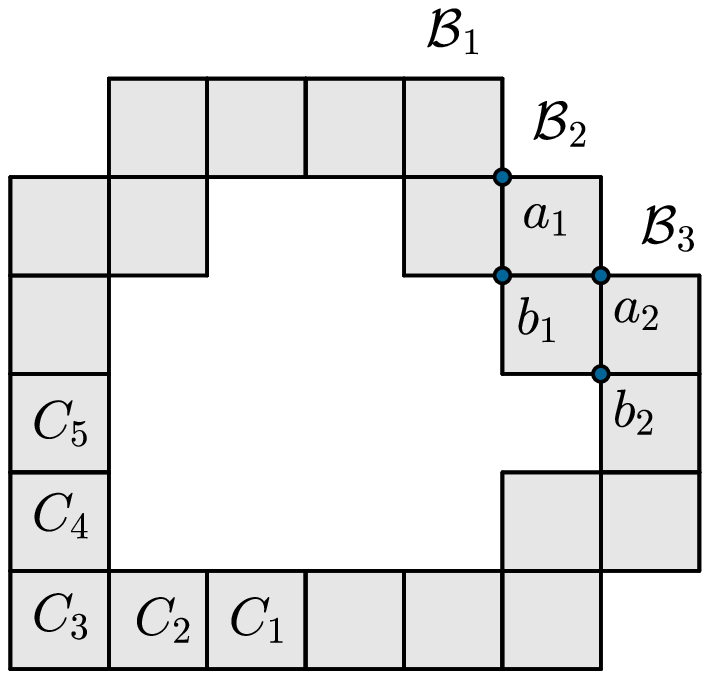}
 	\caption{}
 	\label{Figura:L conf + Ladder}
 \end{figure}

 	\noindent Let $\cP$ be a polyomino. We set $S=K[x_v| v\in V(\cP)]$, where $K$ is a field. If $[a,b]$ is an inner interval of $\cP$, with $a$,$b$ and $c$,$d$ respectively diagonal and anti-diagonal corners, then the binomial $x_ax_b-x_cx_d$ is called an \textit{inner 2-minor} of $\cP$. We define $I_{\cP}$ as the ideal in $S$ generated by all the inner 2-minors of $\cP$ and we call it the \textit{polyomino ideal} of $\cP$. We set also $K[\cP] = S/I_{\cP}$, which is the \textit{coordinate ring} of $\cP$. Eventually, we recall that if $\cP$ is a closed path then having no zig-zag walks, equivalently $\cP$ contains an L-configuration or a ladder of at least three steps, is a necessary and sufficient condition in order to $K[\cP]$ is a domain by [\cite{Cisto_Navarra}, Theorem 6.2].  

\section{Preliminary results}
\label{Section: Properties}
\noindent Let $\cP$ be a non-empty collection of cells with $V(\cP)=\{a_{1},\ldots,a_n\}
$. We define a \emph{$\mathsf{P}$-order} to be a total order on the set $V(\cP)$. Observe that the monomial orderings defined in \cite{Trento2} and \cite{Qureshi} are induced by specific $\mathsf{P}$-orders, in particular we recall the monomial order $<^1$ introduced in \cite{Qureshi}, which will be useful for this paper: we say the $a<^1 b$ if and only if, for $a = ( i , j )$ and $b = ( k , l )$, $i < k$, or $i = k$ and $j < l$.\\
If $<^{\mathsf{P}}$ is a $\mathsf{P}$-order,  we denote by $<^{\mathsf{P}}_{\mathrm{lex}}$ the lexicographic order induced by $<^\mathsf{P}$ on $S=K[x_v \mid v\in V(\cP)]$, that is the lexicographic order induced by the total order on the variables defined in the following way: $x_{a_{i}}\ord x_{a_j}$ if and only if $a_i <^{\mathsf{P}}a_j$ for $i,j\in \{1,\ldots,n\}$. If $f\in S$, we denote by $\lt(f)$ the leading term of $f$ with respect to $\ord$.\\
Let $f,g\in I_{\cP}$, we denote by $S(f,g)$ the $S$-polynomial of $f,g$ with respect to $\ord$. Let $\mathcal{G}$ be the set of all inner 2-minors of $\cP$ (that is the set of generators of $I_{\cP}$). We want to study some conditions on $<^{\mathsf{P}}$ in order to $S(f,g)$ reduces to 0 modulo $\mathcal{G}$. \\
First of all observe that if $[a,b]$ and $[\alpha,\beta]$ are two inner intervals and $[a,b]\cap [\alpha,\beta]$ does not contain any corner of $[a,b]$ and $[\alpha,\beta]$, then $\gcd(\lt(f_{a,b}),\lt(f_{\alpha,\beta}))=1$ so $S(f_{a,b},f_{\alpha,\beta})$ reduces to 0. So it suffices to study the remaining cases.\\
\noindent In the remainder of this section the inner intervals $[a,b]$ and $[\alpha,\beta]$ have respectively $c,d$ and $\gamma,\delta$ as anti-diagonal corners, as in figure~\ref{intervalli}.

\begin{figure}[h]
	\includegraphics[scale=0.6]{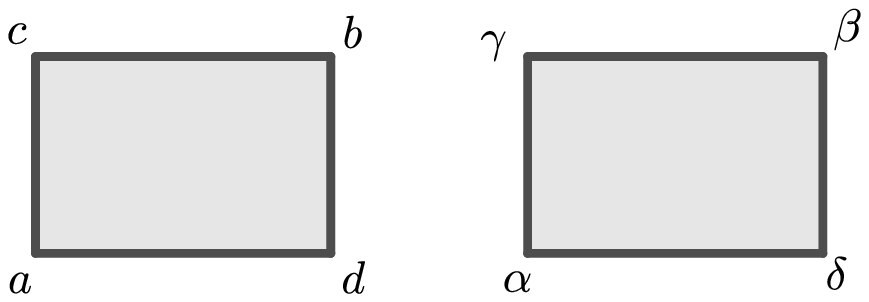}
	\caption{}
	\label{intervalli}
\end{figure}

\noindent In the following lemmas we examine all possible cases in which $|\{a,b,c,d\}\cap \{\alpha,\beta,\gamma,\delta\}|$ is equal to $1$ or $2$.

\begin{lemma}\label{intervalli2}
	Let $\cP$ be a collection of cells and $[a,b]$ and $[\alpha,\beta]$ be two inner intervals such that $|\{a,b,c,d\}\cap \{\alpha,\beta,\gamma,\delta\}|=2$. Then $S(f_{a,b},f_{\alpha,\beta})$ reduces to 0 modulo $\mathcal{G}$ with respect to $\ord$ for any $\mathsf{P}$-order. 
\end{lemma}

\begin{proof}
	We may assume that $\alpha=d$ and $\gamma=b$, because the other cases can be discussed similarly.
	If $\gcd(\lt(f_{a,b}),\lt(f_{\alpha,\beta}))=1$, then there is nothing to prove, so we have to distinguish the following cases.\\
	\underline{First case:} $\lt(f_{a,b})=x_a x_b$ and $\lt(f_{\alpha,\beta})=-x_b x_\delta$. In such a case $S(f_{a,b},f_{\alpha,\beta})=-x_\delta x_d x_c+x_a x_\beta x_d= x_d(x_a x_\beta-x_c x_\delta)$. Observe that $f_{a,\beta}\in I_{\cP}$. If $\lt(S(f_{a,b},f_{\alpha,\beta}))=x_a x_\beta x_d$ then $\lt(f_{a,\beta})=x_a x_\beta$, so $S(f_{a,b},f_{\alpha,\beta})$ reduces to 0. If $\lt(S(f_{a,b},f_{\alpha,\beta}))=-x_\delta x_d x_c$ then $\lt(f_{a,\beta})=-x_c x_\delta$, so $S(f_{a,b},f_{\alpha,\beta})$ reduces to 0 also in this case. \\
	\underline{Second case:} $\lt(f_{a,b})=-x_d x_c$ and $\lt(f_{\alpha,\beta})=x_\beta x_d$. In such a case $S(f_{a,b},f_{\alpha,\beta})=-x_\beta x_a x_b+x_c x_b x_\delta= -x_b(x_a x_\beta-x_c x_\delta)$. As in the first case,  if $\lt(S(f_{a,b},f_{\alpha,\beta}))=-x_\beta x_a x_b$ then $\lt(f_{a,\beta})=x_a x_\beta$, so $S(f_{a,b},f_{\alpha,\beta})$ reduces to 0. If $\lt(S(f_{a,b},f_{\alpha,\beta}))=+x_c x_b x_\delta$ then $\lt(f_{a,\beta})=-x_c x_\delta$, so $S(f_{a,b},f_{\alpha,\beta})$ reduces to 0 also in this case.
\end{proof}

%
%
%

\begin{lemma}
	Let $\cP$ be a collection of cells and $[a,b]$ and $[\alpha,\beta]$ be two inner intervals with $\beta=b$ and $\gamma\in ]c,b[$ (see Figure~\ref{img_intervalli4}(A)). Let $h$ be the vertex such that $[h,b]$ is the inner interval having $d,\gamma$ as anti-diagonal corner and $r$ be the vertex such that $[r,h]$ is the interval having $a,\alpha$ as anti-diagonal corner. Let $<^\mathsf{P}$ be a $\mathsf{P}$-order on $V(\cP)$ and suppose that $\gcd(\lt(f_{a,b}),\lt(f_{\alpha,\beta}))\neq 1$. Then $S(f_{a,b},f_{\alpha,\beta})$ reduces to 0 modulo $\mathcal{G}$ with respect to $\ord$ if and only if one of the following conditions occurs:
	\begin{enumerate}
		\item $x_a x_\gamma x_\delta \ord x_\alpha x_c x_d$ and in addiction $h,\delta <^{\mathsf{P}}\alpha$ or $h,\delta <^{\mathsf{P}} d$;
		\item $x_a x_\gamma x_\delta \ord x_\alpha x_c x_d$, $\{r,h,a,\alpha\}$ is the set of vertices of an inner interval of $\cP$ and in addiction $r,\gamma <^{\mathsf{P}}\alpha$ or $r,\gamma <^{\mathsf{P}} c$;
		\item $x_\alpha x_c x_d \ord x_a x_\gamma x_\delta$ and in addiction $h,c <^{\mathsf{P}}a$ or $h,c <^{\mathsf{P}} \gamma$;
		\item $x_\alpha x_c x_d \ord x_a x_\gamma x_\delta$, $\{r,h,a,\alpha\}$ is the set of vertices of an inner interval of $\cP$ and in addiction $r,d <^{\mathsf{P}}a$ or $r,d <^{\mathsf{P}} \delta$. 
	\end{enumerate}
The same characterization holds for $S(f_{c,d},f_{\gamma,\delta})$, $S(f_{b,a},f_{b,\alpha})$ and $S(f_{d,c},f_{\delta,\gamma})$ considering all the rotations of the described configuration (see respectively Figure~\ref{img_intervalli4}(B), Figure~\ref{img_intervalli4}(C) and Figure~\ref{img_intervalli4}(D)). 	
	\label{intervalli4}
\end{lemma}

\begin{figure}[h]
\centering
	\subfloat[]{\includegraphics[scale=0.65]{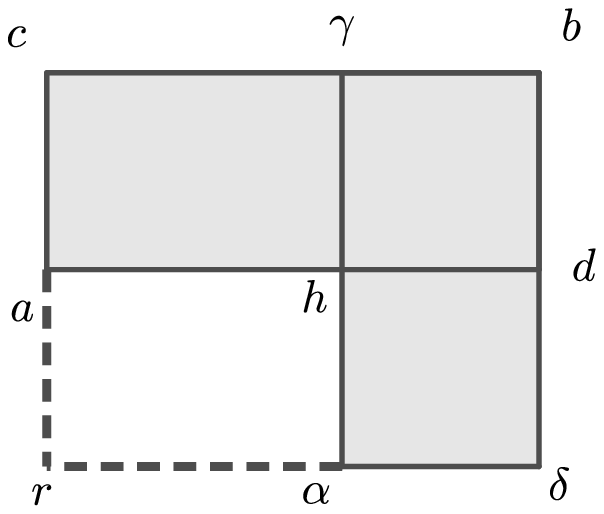}}\quad 
	\subfloat[]{\includegraphics[scale=0.65]{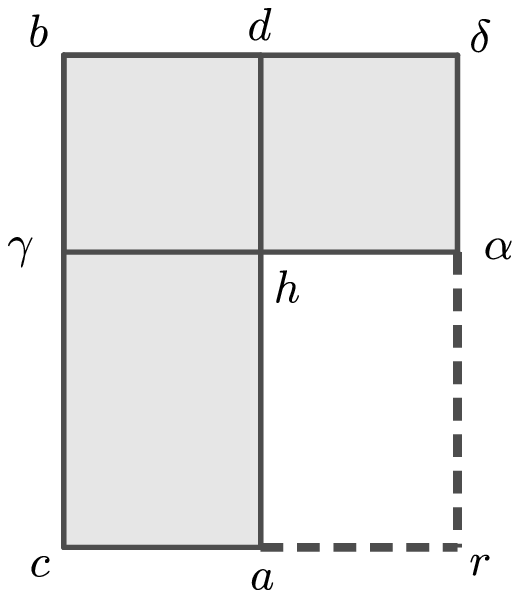}}\quad
	\subfloat[]{\includegraphics[scale=0.65]{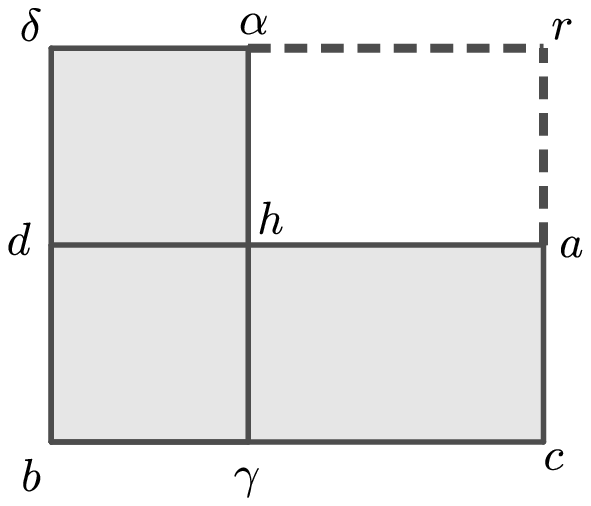}}\quad
	\subfloat[]{\includegraphics[scale=0.65]{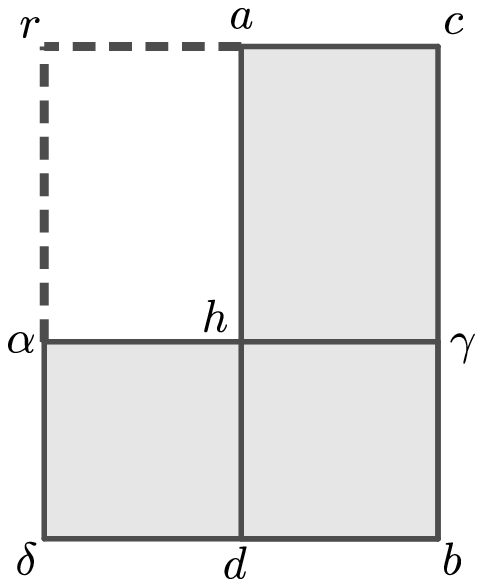}}
	\caption{}
	\label{img_intervalli4}
\end{figure}

\begin{proof}
	Observe that $\gcd(\lt(f_{a,b}),\lt(f_{\alpha,\beta}))\neq 1$ if and only if $\lt(f_{a,b})=x_a x_b$ and $\lt(f_{\alpha,\beta})=x_\alpha x_b$. Since $S(f_{a,b},f_{\alpha,\beta})=-x_\alpha x_c x_d+x_a x_\gamma x_\delta$, we have two possibilities:\\
	1) $\lt(S(f_{a,b},f_{\alpha,\beta}))=-x_\alpha x_c x_d$, in particular $x_a x_\gamma x_\delta \ord x_\alpha x_c x_d$. Observe that, since $x_c x_d$ is not the leading term of $f_{a,b}$, in such a case the only possibilities for the reduction of $S(f_{a,b},f_{\alpha,\beta})$ is through a first division by $f_{\alpha,d}$ if $\lt(f_{\alpha,d})=x_\alpha x_d$ or by $f_{r \gamma}$ if $\lt(f_{r,\gamma})=-x_\alpha x_c$. The first case is possible if and only if $(h,\delta <^{\mathsf{P}}\alpha)\vee (h,\delta <^{\mathsf{P}} d)$, and in such case indeed, after a little computation, $S(f_{a,b},f_{\alpha,\beta})$ reduces by $f_{\alpha,d}$ to $x_\delta(x_a x_\gamma-x_c x_h)=x_\delta f_{a\gamma}$ and this one reduces to 0. For this case we obtain the condition (1) of this lemma. The second case is possible if and only if the condition (2) is satisfied, that is if $[r,h]$ is an inner interval of $\cP$ and $(r,\gamma <^{\mathsf{P}}\alpha)\vee (r,\gamma <^{\mathsf{P}} c)$. In such a case in fact $S(f_{a,b},f_{\alpha,\beta})$ reduces through $f_{r,\gamma}$ to $x_\gamma(x_a x_\delta-x_r x_d)=x_\gamma f_{a\delta}$ and this one reduces to 0.\\
	2) $\lt(S(f_{a,b},f_{\alpha,\beta}))=x_a x_\gamma x_\delta$, in particular $x_\alpha x_c x_d \ord x_a x_\gamma x_\delta$. We can argue as in the first part of this proof observing that, since $x_\gamma x_\delta$ is not the leading term of $f_{\alpha,\beta}$, in such a case the only possibilities for the reduction of $S(f_{a,b},f_{\alpha,\beta})$ is through $f_{a,\gamma}$ if $\lt(f_{a,\gamma})=x_a x_\gamma$ or by $f_{r,d}$ if $\lt(f_{r,d})=-x_a x_\delta$. The first case is possible if and only if $(h,c <^{\mathsf{P}}a)\vee (h,c <^{\mathsf{P}} \gamma)$ , that is the condition (3) holds, while the second is possible if and only if $[r,h]$ is an inner interval of $\cP$ and $(r,d <^{\mathsf{P}}a)\vee (r,d <^{\mathsf{P}} \delta)$, that is the condition (4) is satisfied. In both cases $S(f_{a,b},f_{\alpha,\beta})$ reduces to $0$. \\ 
The last statement of this lemma is verified since the only effects of the rotation of a configuration are the different notations for the same intervals (for instance $[a,b]$ becomes $[c,d]$, $[b,a]$ or $[d,c]$) or the change of the sign of the binomials in the generators of $I_{\cP}$.
\end{proof}

\noindent The following four lemmas can be proved by the same arguments of Lemma~\ref{intervalli4}, so we omit their proofs.

\begin{lemma}
	Let $\cP$ be a collection of cells and $[a,b]$ and $[\alpha,\beta]$ be two inner intervals with $\gamma=b$ and $\alpha \in ]d,b[$ (see Figure~\ref{img_intervalli5}(A)). Let $h$ be the vertex such that $[h,b]$ is the inner interval having $c,\alpha$ as anti-diagonal corner and $r$ be the vertex such that $r,\alpha$ are the anti-diagonal corners of the interval $[d,\delta]$. Let $<^\mathsf{P}$ be a $\mathsf{P}$-order on $V(\cP)$ and suppose that $\gcd(\lt(f_{a,b}),\lt(f_{\alpha,\beta}))\neq 1$. Then $S(f_{a,b},f_{\alpha,\beta})$ reduces to 0 modulo $\mathcal{G}$ with respect to $\ord$ if and only if one of the following conditions occurs:
	\begin{enumerate}
		\item $x_a x_\alpha x_\beta \ord x_\delta x_c x_d$ and in addiction $h,\beta <^{\mathsf{P}}c$ or $h,\beta <^{\mathsf{P}} \delta$;
		\item $x_a x_\alpha x_\beta \ord x_\delta x_c x_d$, $\{d,\delta,\alpha,r\}$ is the set of vertices of an inner interval of $\cP$ and in addiction $r,\alpha <^{\mathsf{P}}\delta$ or $r,\alpha <^{\mathsf{P}} d$;
		\item $x_\delta x_c x_d \ord x_a x_\alpha x_\beta$ and in addiction $h,d <^{\mathsf{P}}a$ or $h,d <^{\mathsf{P}} \alpha$;
		\item $x_\delta x_c x_d \ord x_a x_\alpha x_\beta$, $\{d,\delta,\alpha,r\}$ is the set of vertices of an inner interval of $\cP$ and in addiction $r,c <^{\mathsf{P}}a$ or $r,c <^{\mathsf{P}} \beta$.
	\end{enumerate}
The same characterization holds for $S(f_{c,d},f_{\gamma,\delta})$, $S(f_{b,a},f_{\beta,\alpha})$ and $S(f_{d,c},f_{\delta,\gamma})$ considering all the rotations of the described configuration (see respectively Figure~\ref{img_intervalli5}(B), Figure~\ref{img_intervalli5}(C) and Figure~\ref{img_intervalli5}(D)).	
	\label{intervalli5}
\end{lemma}

\begin{figure}[h]
	\centering
	\subfloat[]{\includegraphics[scale=0.65]{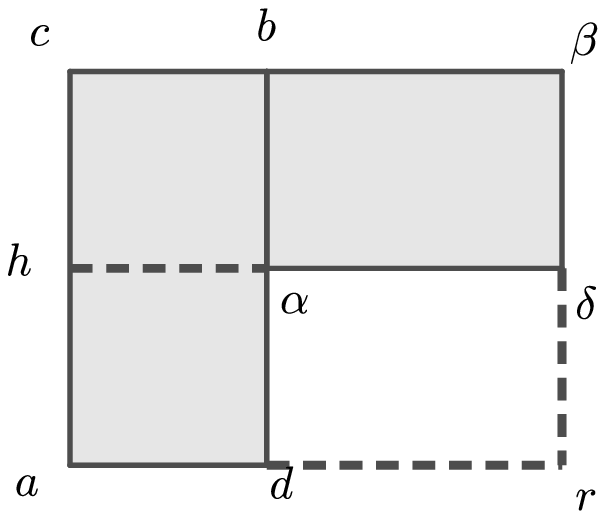}}\quad 
	\subfloat[]{\includegraphics[scale=0.65]{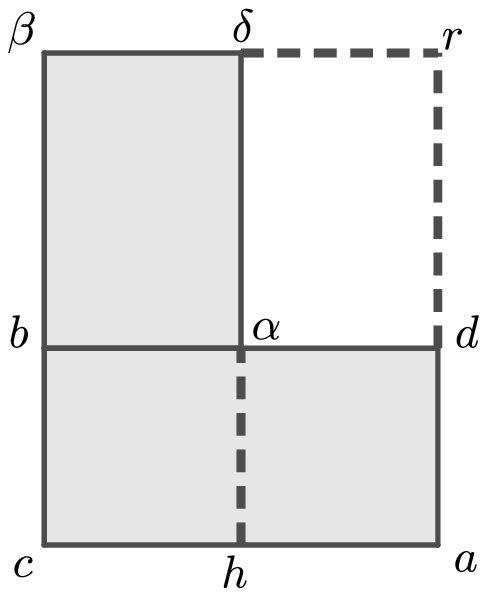}}\quad
	\subfloat[]{\includegraphics[scale=0.65]{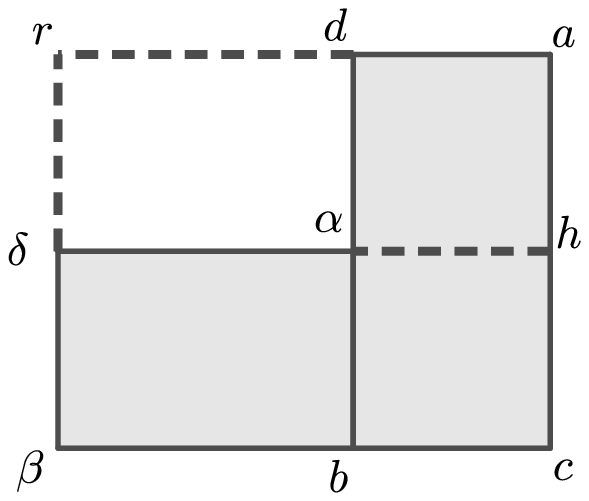}}\quad
	\subfloat[]{\includegraphics[scale=0.65]{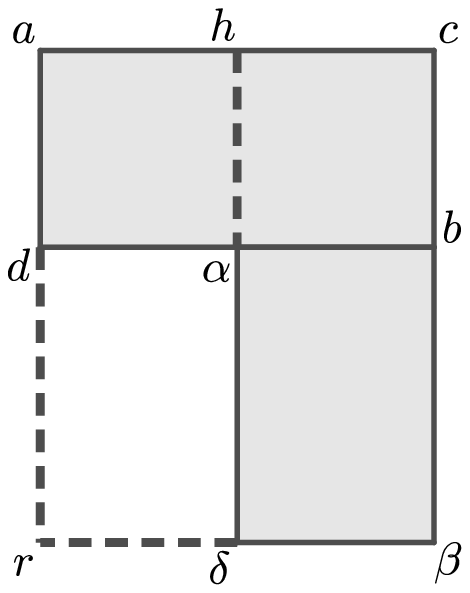}}
	\caption{}
	\label{img_intervalli5}
\end{figure}


\begin{lemma}
	Let $\cP$ be a collection of cells and $[a,b]$ and $[\alpha,\beta]$ be two inner intervals with $\alpha=c$ and $b \in ]\alpha,\delta[$ (see Figure~\ref{img_intervalli6}(A)). Let $h$ be the vertex such that $h,\delta$ are the diagonal corners of the inner interval $[b,\beta]$ and $r$ be the vertex such that $r,b$ are the anti-diagonal corners of the interval $[d,\delta]$. Let $<^\mathsf{P}$ be a $\mathsf{P}$-order on $V(\cP)$ and suppose that $\gcd(\lt(f_{a,b}),\lt(f_{\alpha,\beta}))\neq 1$. Then $S(f_{a,b},f_{\alpha,\beta})$ reduces to 0 modulo $\mathcal{G}$ with respect to $\ord$ if and only if one of the following conditions occurs:
	\begin{enumerate}
		\item $x_d x_\delta x_\gamma \ord x_\beta  x_a x_b$ and in addiction $h,\delta <^{\mathsf{P}}b$ or $h,\delta <^{\mathsf{P}} \beta$;
		\item $x_d x_\delta x_\gamma \ord x_\beta  x_a x_b$, $\{d,\delta,b,r\}$ is the set of vertices of an inner interval of $\cP$ and in addiction $r,\gamma <^{\mathsf{P}}a$ or $r,\gamma <^{\mathsf{P}} \beta$;
		\item $x_\beta  x_a x_b \ord x_d x_\delta x_\gamma$ and in addiction $h,a <^{\mathsf{P}}d$ or $h,a <^{\mathsf{P}} \gamma$;
		\item $x_\beta  x_a x_b \ord x_d x_\delta x_\gamma$, $\{d,\delta,b,r\}$ is the set of vertices of an inner interval of $\cP$ and in addiction $r,b <^{\mathsf{P}}d$ or  $r,b <^{\mathsf{P}} \delta$.
	\end{enumerate}
The same characterization holds for $S(f_{c,d},f_{\gamma,\delta})$, $S(f_{b,a},f_{\beta,\alpha})$ and $S(f_{d,c},f_{\delta,\gamma})$ considering all the rotations of the described configuration (see respectively Figure~\ref{img_intervalli6}(B), Figure~\ref{img_intervalli6}(C) and Figure~\ref{img_intervalli6}(D)).		
	\label{intervalli6}
\end{lemma}

\begin{figure}[h]
	\centering
	\subfloat[]{\includegraphics[scale=0.65]{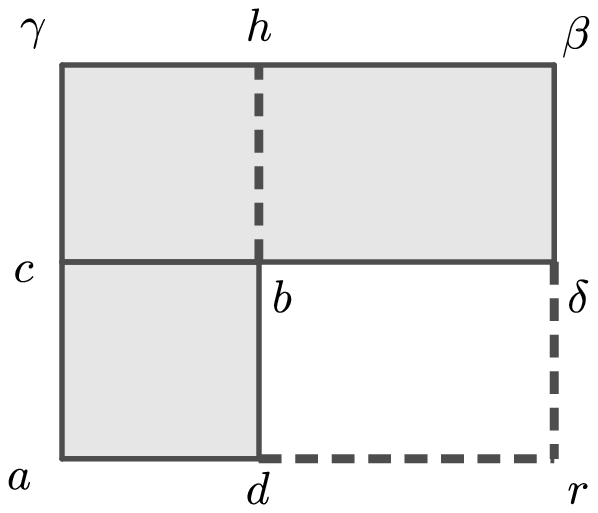}}\quad 
	\subfloat[]{\includegraphics[scale=0.65]{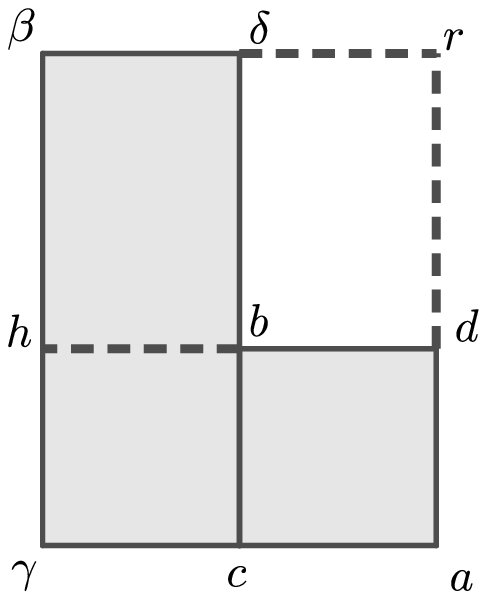}}\quad
	\subfloat[]{\includegraphics[scale=0.65]{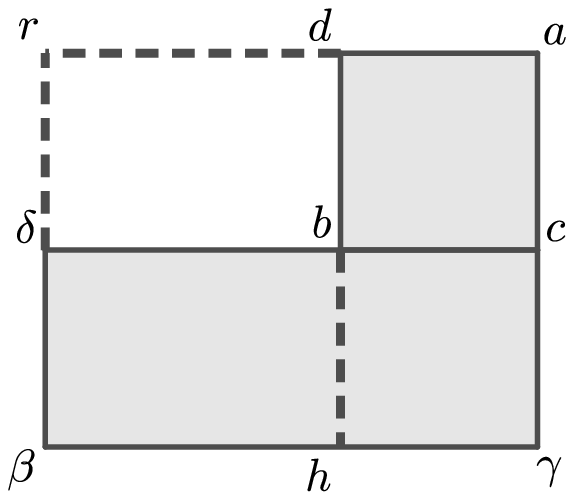}}\quad
	\subfloat[]{\includegraphics[scale=0.65]{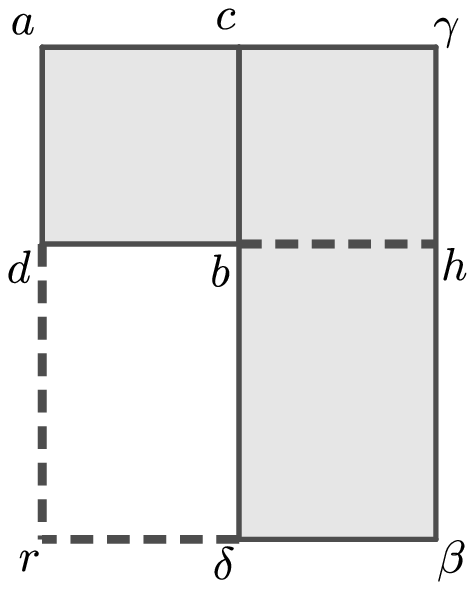}}
	\caption{}
	\label{img_intervalli6}
\end{figure}


\begin{lemma}
	Let $\cP$ be a collection of cells and $[a,b]$ and $[\alpha,\beta]$ be two inner intervals with $\gamma=c$ and $\delta \in ]a,b[$ (see Figure~\ref{img_intervalli7}(A)). Let $[h,r]$ be the inner interval having  $d,\delta$ as anti-diagonal corners. Let $<^\mathsf{P}$ be a $\mathsf{P}$-order on $V(\cP)$ and suppose that $\gcd(\lt(f_{a,b}),\lt(f_{\alpha,\beta}))\neq 1$. Then $S(f_{a,b},f_{\alpha,\beta})$ reduces to 0 modulo $\mathcal{G}$ with respect to $\ord$ if and only if one of the following conditions occurs:
	\begin{enumerate}
		\item $x_d x_\alpha x_\beta \ord x_\delta x_a x_b$ and in addiction $h,\alpha <^{\mathsf{P}}a$ or $h,\alpha <^{\mathsf{P}} \delta$;
		\item $x_d x_\alpha x_\beta \ord x_\delta x_a x_b$ and in addiction $r,\beta <^{\mathsf{P}}\delta$ or $r,\beta <^{\mathsf{P}} b$;
		\item $x_\delta x_a x_b \ord x_d x_\alpha x_\beta$ and in addiction $r,a <^{\mathsf{P}}\alpha$ or $r,a <^{\mathsf{P}} d$;
		\item $x_\delta x_a x_b \ord x_d x_\alpha x_\beta$ and in addiction $h,b <^{\mathsf{P}}d$ or $h,b <^{\mathsf{P}} \beta$.
	\end{enumerate}
The same characterization holds for $S(f_{c,d},f_{\gamma,\delta})$, $S(f_{b,a},f_{\beta,\alpha})$ and $S(f_{d,c},f_{\delta,\gamma})$ considering all the rotations of the described configuration (see respectively Figure~\ref{img_intervalli7}(B), Figure~\ref{img_intervalli7}(C) and Figure~\ref{img_intervalli7}(D)).		
	\label{intervalli7}
\end{lemma}

\begin{figure}[h]
	\centering
	\subfloat[]{\includegraphics[scale=0.65]{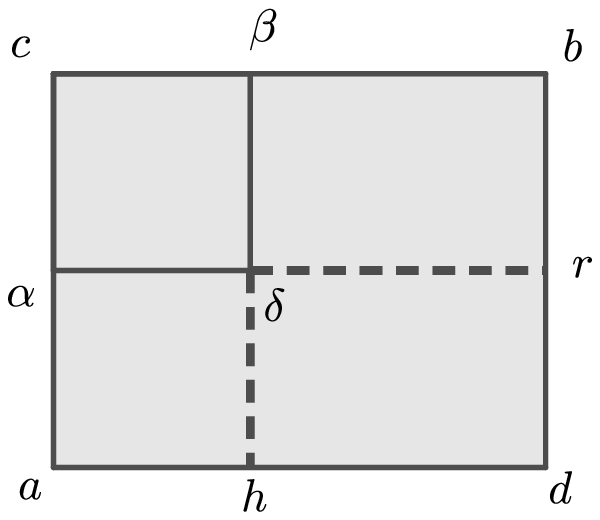}}\quad 
	\subfloat[]{\includegraphics[scale=0.65]{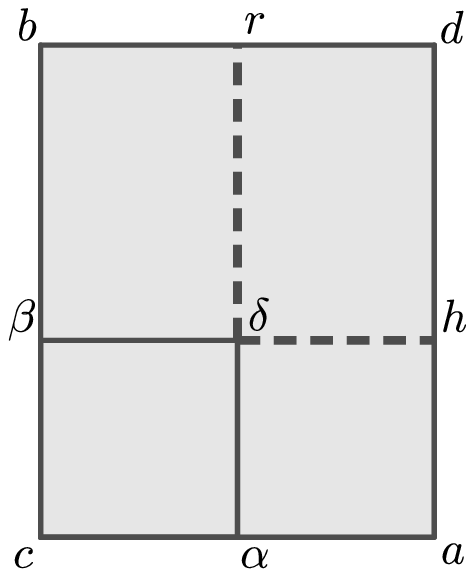}}\quad
	\subfloat[]{\includegraphics[scale=0.65]{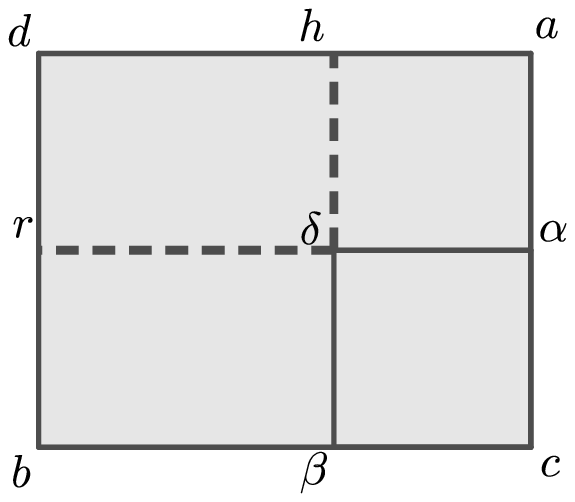}}\quad
	\subfloat[]{\includegraphics[scale=0.65]{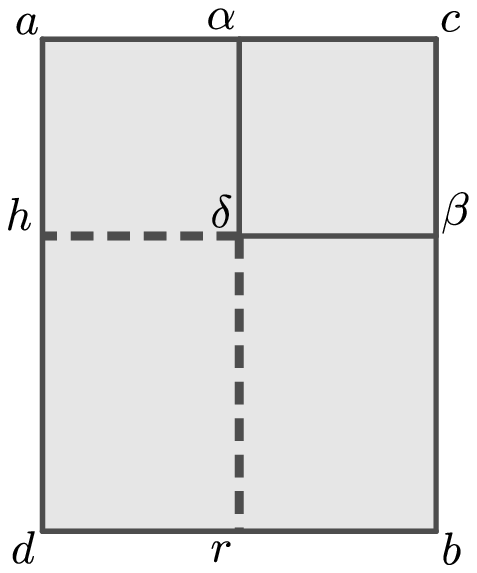}}
	\caption{}
	\label{img_intervalli7}
\end{figure}


\begin{lemma}
	Let $\cP$ be a collection of cells and $[a,b]$ and $[\alpha,\beta]$ be two inner intervals with $\alpha=b$ and $\beta \notin [a,b]$ (see Figure~\ref{img_intervalli8}(A)). Let $h,r$ be the anti-diagonal corners,  different to $b$, respectively of the intervals $[d,\delta]$ and $[c,\gamma]$. Let $<^\mathsf{P}$ be a $\mathsf{P}$-order on $V(\cP)$ and suppose that $\gcd(\lt(f_{a,b}),\lt(f_{\alpha,\beta}))\neq 1$. Then $S(f_{a,b},f_{\alpha,\beta})$ reduces to 0 modulo $\mathcal{G}$ with respect to $\ord$ if and only if one of the following conditions occurs:
	\begin{enumerate}
\item $x_a x_\gamma x_\delta \ord x_\beta x_d x_c$, $\{d,\delta,b,h\}$ is the set of vertices of an inner interval of $\cP$ and in addiction $h,\gamma <^{\mathsf{P}}\beta$ or $h,\gamma <^{\mathsf{P}} d$;
	\item $x_a x_\gamma x_\delta \ord x_\beta x_d x_c$, $\{c,\gamma,b,r\}$ is the set of vertices of an inner interval of $\cP$ and in addiction $r,\delta <^{\mathsf{P}}\beta$ or $r,\delta <^{\mathsf{P}} c$;
		\item $x_\beta x_d x_c \ord x_a x_\gamma x_\delta$, $\{c,\gamma,b,r\}$ is the set of vertices of an inner interval of $\cP$ and in addiction $r,d <^{\mathsf{P}}a)$ or $r,d <^{\mathsf{P}} \gamma$;
		\item $x_\beta x_d x_c \ord x_a x_\gamma x_\delta$, $\{d,\delta,b,h\}$ is the set of vertices of an inner interval of $\cP$ and in addiction $c,h <^{\mathsf{P}}a$ or $c,h <^{\mathsf{P}} \delta$.
	\end{enumerate}
The same characterization holds for $S(f_{c,d},f_{\gamma,\delta})$, $S(f_{b,a},f_{\beta,\alpha})$ and $S(f_{d,c},f_{\delta,\gamma})$ considering all the rotations of the described configuration (see respectively Figure~\ref{img_intervalli8}(B), Figure~\ref{img_intervalli8}(C) and Figure~\ref{img_intervalli8}(D)).		
	\label{intervalli8}
\end{lemma}

\begin{figure}[h]
	\centering
	\subfloat[]{\includegraphics[scale=0.65]{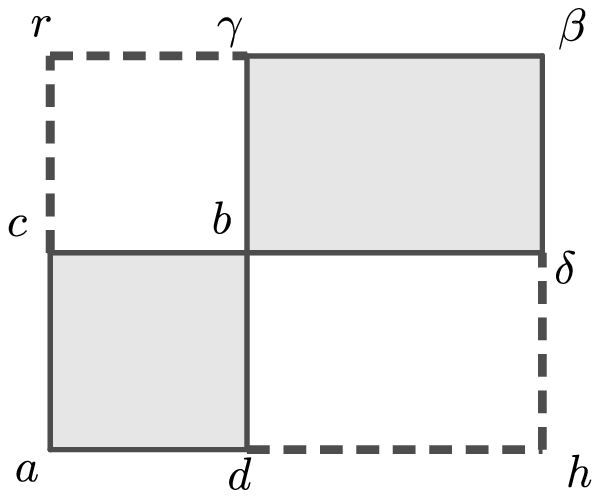}}\quad 
	\subfloat[]{\includegraphics[scale=0.65]{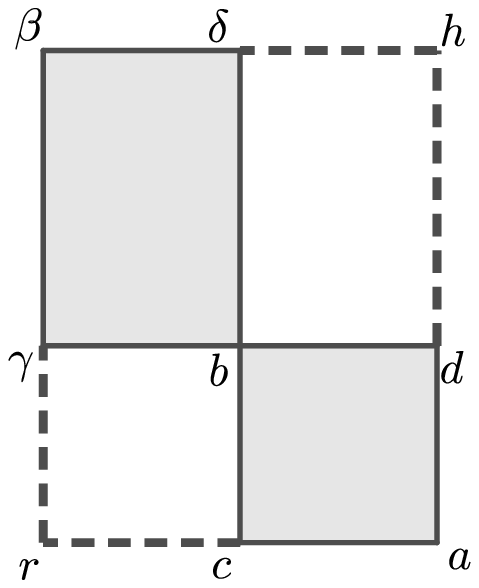}}\quad
	\subfloat[]{\includegraphics[scale=0.65]{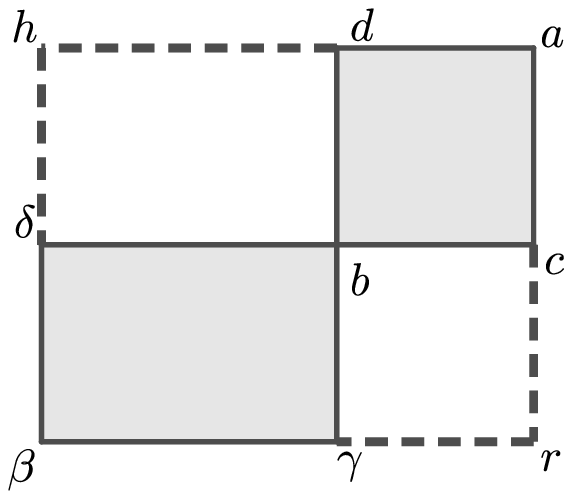}}\quad
	\subfloat[]{\includegraphics[scale=0.65]{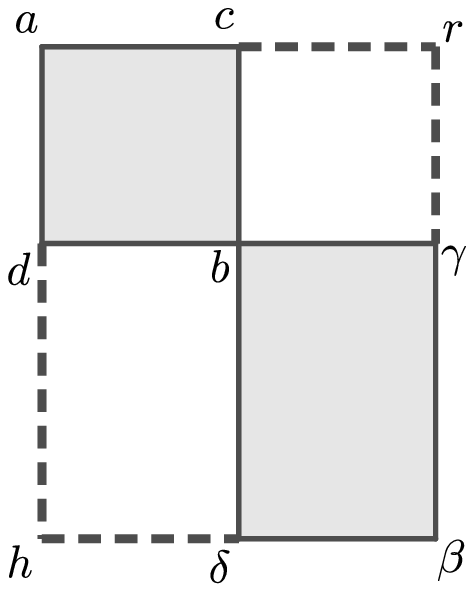}}
	\caption{}
	\label{img_intervalli8}
\end{figure}



\section{Gr\"obner basis of the polyomino ideal of a closed path.}\label{Section: Grobner basis of polyomino ideal of closed path}

\noindent Let $\cP$ be a polyomino. In this section we examine four special configurations of cells of a polyomino, that permit us when $\cP$ is a closed path to provide some particular subsets $Y\subset V(\cP)$ for which we can define the following $\mathsf{P}$-order.

\begin{defn}\rm
	Let $Y\subset V(\cP)$. We define the $\mathsf{P}$-order $<^Y$ in the following way:
	\[
	a<^Y b \Leftrightarrow
	\left\{
	\begin{array}{l}
	a\notin Y\ \mbox{and}\ b\in Y \\
	a,b\notin Y\ \mbox{and}\ a<^1 b \\
	a,b\in Y\ \mbox{and}\ a<^1 b
	\end{array}
	\right.
	\]
for $a,b \in V(\cP)$.
\end{defn}

\noindent We call an \textit{W-pentomino with middle cell $A$} a collection of cells of $\cP$ consisting of an horizontal block $\cB_1=[A_1,B_1]$ of rank two, a vertical block $\cB_2=[A_2,B_2]$ of rank two and a cell $A$ not belonging to $\cB_{1}\cup\cB_2$, such that $V(\cB_1)\cap V(\cB_2)=\{w\}$ and where $w$ is the lower right corner of $A$. Moreover, if $\mathcal{W}$ is a  W-pentomino with middle cell $A$, we denote with $x_W$ the left upper corner of $A$, with $y_W$ the lower right corner of $B_1$ and with $z_W$ the lower right corner of $A_2$. See Figure~\ref{Figura:pentomino}.

\begin{figure}[h]
\centering
\subfloat[]{\includegraphics[scale=0.8]{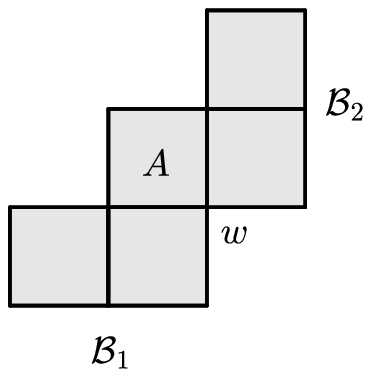}}\qquad \qquad
\subfloat[]{\includegraphics[scale=0.8]{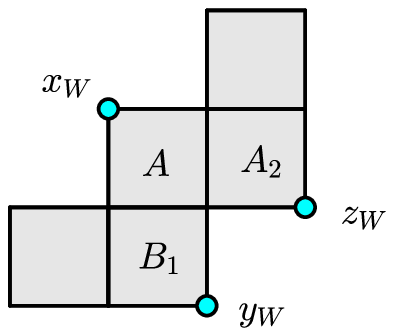}}
\caption{W-pentomino}
\label{Figura:pentomino}
\end{figure}  
 
\noindent We call an \textit{LD-horizontal (vertical) skew tetromino} a collection of cells of $\cP$ consisting of two horizontal (vertical) blocks of rank two $\cB_1=[A_1,B_1]$ and $\cB_2=[A_2,B_2]$ such that $V(B_1)\cap V(A_2)=\{w_1,w_2\}$ and $w_1,w_2$ are right and left upper (lower and upper right) corners of $B_1$. Moreover, if $\mathcal{C}$ is an LD-horizontal (vertical) skew tetromino, we denote with $x_\cC,y_\cC$ the left and right upper corners of $A_2$ (the upper and lower left corners of $B_1$), and with $a_\cC,b_\cC$ the left and right lower corners of $B_1$ (the upper and lower right corners of $A_2$). See Figure~\ref{Figura:skew-tetromino}.

\begin{figure}[h]
	\centering
	\subfloat[]{\includegraphics[scale=0.8]{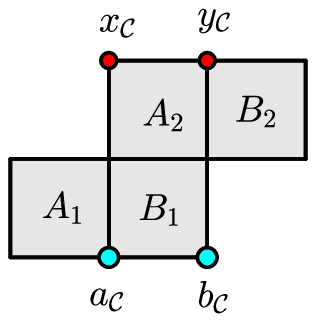}}\qquad \qquad
	\subfloat[]{\includegraphics[scale=0.8]{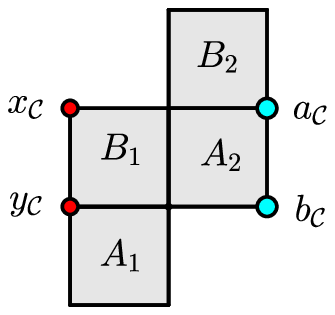}}
	\caption{LD-horizontal skew tetromino (A) and LD-vertical skew tetromino (B)}
	\label{Figura:skew-tetromino}
\end{figure}

\noindent We call an \textit{LD-horizontal (vertical) skew hexomino} a collection of cells of $\cP$ consisting of two horizontal (vertical) blocks of rank three $\cB_1=[A_1,B_1]$ and $\cB_2=[A_2,B_2]$ such that $V(B_1)\cap V(A_2)=\{w_1,w_2\}$ and $w_1,w_2$ are respectively the right and left upper (lower and upper right) corners of $B_1$. Moreover, if $\mathcal{D}$ is an LD-horizontal (vertical) skew tetromino, we denote by $x_D,y_D$ the left and right upper corners of $A_2$ (the upper and lower left corners of $B_1$), and by $a_D,b_D$ the the left and right upper corners of $B_1$ (the upper and lower right corners of $A_2$). See Figure~\ref{Figura:skew-hexomino}.

\begin{figure}[h]
	\centering
	\subfloat[]{\includegraphics[scale=0.8]{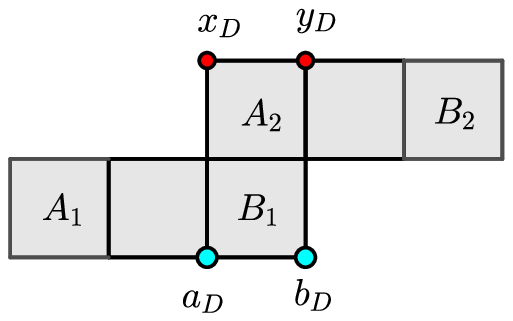}}\qquad \qquad
	\subfloat[]{\includegraphics[scale=0.8]{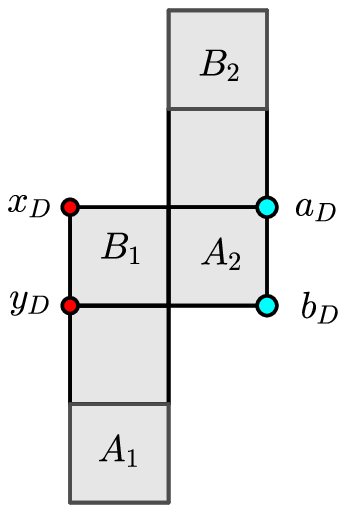}}
	\caption{LD-horizontal skew hexomino (A) and LD-vertical skew hexomino (B)}
	\label{Figura:skew-hexomino}
\end{figure}

\noindent We call an \textit{RW-heptomino with middle cell $A$} a collection of cells of $\cP$ consisting of an horizontal block $\cB_1=[A_1,B_1]$ of rank three, a vertical block $\cB_2=[A_2,B_2]$ of rank three and a cell $A$ not belonging to $\cB_{1}\cup\cB_2$, such that $V(\cB_1)\cap V(\cB_2)=\{w\}$ and where $w$ is the upper left corner of $A$. Moreover, if $\mathcal{T}$ is an RW-pentomino with middle cell $A$, we denote by $x_T$ the right lower corner of $A$, with $y_T$ the left upper corner of $B_2$ and by $z_T$ the left upper corner of $A_1$. See Figure~\ref{Figura:heptomino}.

\begin{figure}[h]
	\centering
	\subfloat[]{\includegraphics[scale=0.7]{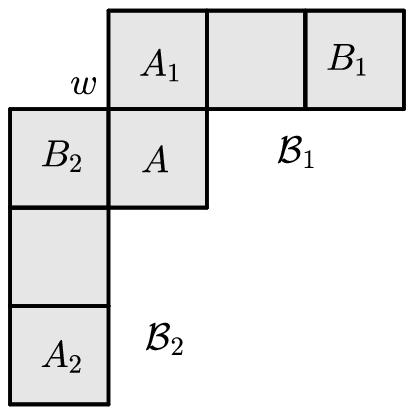}}\qquad \qquad
	\subfloat[]{\includegraphics[scale=0.7]{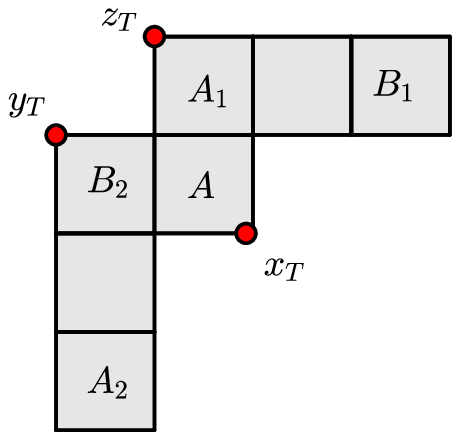}}
	\caption{RW-heptomino}
	\label{Figura:heptomino}
\end{figure}


\begin{thm}
Let $\cP$ be a closed path polyomino not containing any W-pentomino. Let $\mathcal{R}$ be the set of all LD-horizontal and vertical skew tetromino contained in $\cP$ and let $Y=\bigcup_{\cC\in \mathcal{R}}\{x_\cC,y_\cC\}$. Then $\mathcal{G}$ is the reduced Gr\"obner basis of $I_\cP$ with respect to the monomial order $<^{Y}_\mathrm{lex}$.
\label{without_W}
\end{thm}

\begin{proof}
Let $f=x_px_q-x_rx_s$ and $g=x_ux_v-x_wx_z$ be the two binomials attached respectively to the inner intervals $[p,q]$ and $[u,v]$ of $\cP$. We prove that $S(f,g)$ reduces to $0$ modulo $\cG$ with respect to $<^{Y}_\mathrm{lex}$, examining all possible cases on $\{p,q,r,s\}\cap \{u,v,w,z\}$. \\
The case $\{p,q,r,s\}\cap \{u,v,w,z\}=\emptyset$ is trivial. If $|\{p,q,r,s\}\cap \{u,v,w,z\}|=2$, then the claim follows from Lemma \ref{intervalli2}. Assume that $|\{p,q,r,s\}\cap \{u,v,w,z\}|=1$ and that $[p,q]$ is not contained in $[u,v]$ or vice versa. 
 Suppose that $q=v$. For the structure of $\cP$ we may assume that $s\in ]z,v[$ and $w\in ]r,q[$, so there exists $k\in\{1,\dots,n\}$ such that $A_k=[p,q]\cap [u,v]$. Let $A_{k-1}$ be the cell of $\cP_{[p,q]}$ adjacent to $A_k$. If $A_{k-2}$ is at North of $A_{k-1}$ then we have the conclusion from (1) of Lemma \ref{intervalli4}. If $A_{k-2}$ is at West of $A_{k-1}$ then the claim follows either by being $\gcd(\lt(f),\lt(g))= 1$ or by applying (1) of Lemma \ref{intervalli4} if $\gcd(\lt(f),\lt(g))\neq 1$. The cases $r=w$, $p=u$ and $s=z$ can be proved similarly to the previous ones.
Suppose that $q=w$. We may assume that $u\in]s,q[$, because the arguments are similar when $s\in]u,w[$. Let $A_k$ be the cell of $\cP$ having $r$, $u$ as anti-diagonal corners and we denote by $A_{k-1}$ and $A_{k+1}$ respectively the cells of $\cP_{[p,q]}$ and $\cP_{[u,v]}$ adjacent to $A_k$. If $\{A_{k-2},A_{k-1},A_k,A_{k+1} \}$ is an LD-vertical skew tetromino or $\{A_{k-2},A_{k-1},A_k,A_{k+1},A_{k+2}\}$ is an $L$-configuration then $\gcd(\lt(f),\lt(g))= 1$. If $\{A_{k-1},A_k,A_{k+1},A_{k+2} \}$ is an LD-horizontal skew tetromino, then $\gcd(\lt(f),\lt(g))= x_w$ and applying (1) of Lemma \ref{intervalli5} we have the desired conclusion. Similar arguments hold in the cases $s=u$, $v=r$ and $z=p$. Suppose $q=u$ and let $A_k$ and $A_{k+2}$ be the cells of $\cP$ having respectively $q$ as upper right and lower left corner. If $\{A_{k-1},A_k,A_{k+1},A_{k+2} \}$ or $\{A_{k},A_{k+1},A_{k+2},A_{k+3} \}$ is an LD-vertical skew tetromino, then $\gcd(\lt(f),\lt(g))=1$. If $\{A_{k-1},A_{k},A_{k+1},A_{k+2},A_{k+3} \}$ is an $L$-configuration, the claim follows either by $\gcd(\lt(f),\lt(g))= 1$ or by applying Lemma \ref{intervalli8} if $\gcd(\lt(f),\lt(g))\neq 1$. If $\{A_{k-1},A_{k},A_{k+1},A_{k+2},A_{k+3} \}$ is not an $L$-configuration and does not contain an LD-vertical skew tetromino, then the only two possibilities are that either $\{A_{k-1},A_k,A_{k+1},A_{k+2} \}$ or $\{A_{k},A_{k+1},A_{k+2},A_{k+3} \}$ is an LD-horizontal skew tetromino. In both cases $\gcd(\lt(f),\lt(g))= 1$, in particular in the first case the claim follows since $\cP$ has not any $W$-pentomino, so $A_{k+3}$ is at East of $A_{k+2}$. The other cases $s=w$, $z=r$ or $v=p$ can be proved by similar arguments. Finally, it is easy to see that in such cases $\mathcal{G}$ is also the reduced Gr\"obner basis of $I_\cP$. 
\end{proof}

\noindent In Figure~\ref{Figure:without_W}(A) is shown an example of polyomino satisfying Theorem~\ref{without_W}. 

\begin{rmk}\rm In \cite{Trento2} the authors introduced the class of \emph{thin} polyominoes, that consists of all polyominoes not containing the configuration whose shape is a square made up of four cells. Such class can be viewed as a generalization of closed paths. We observe that the conclusion of the previous theorem does not hold in general for thin polyominoes, using the same monomial order. In fact, we can consider the thin polyomino in Figure~\ref{Figure:without_W}(B) and it is not difficult to show that the S-polynomial associated to the marked intervals does not reduce to 0.
\end{rmk}

\begin{figure}[h]
\centering
\subfloat[]{\includegraphics[scale=0.6]{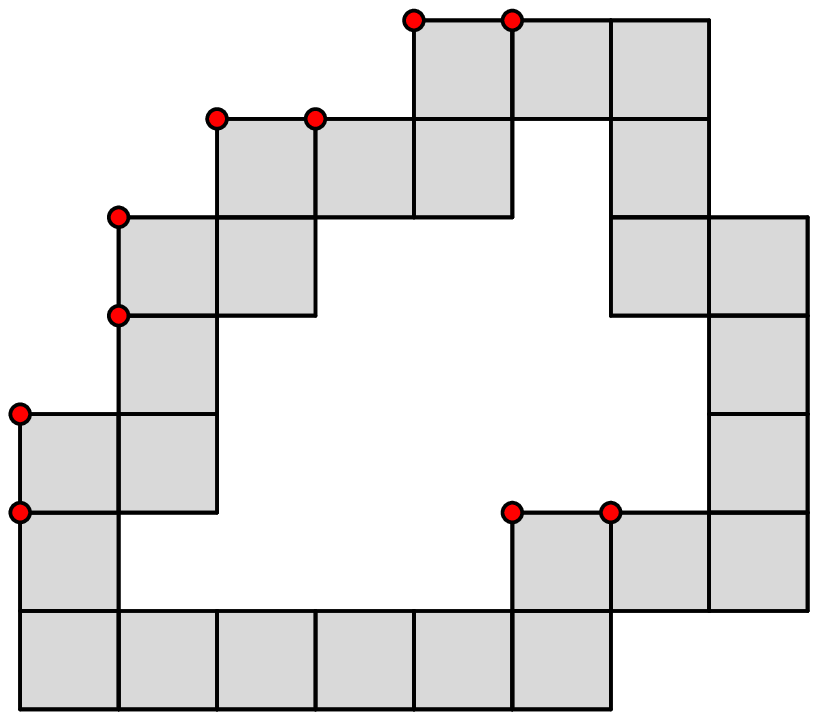}}\qquad \qquad
\subfloat[]{\includegraphics[scale=0.6]{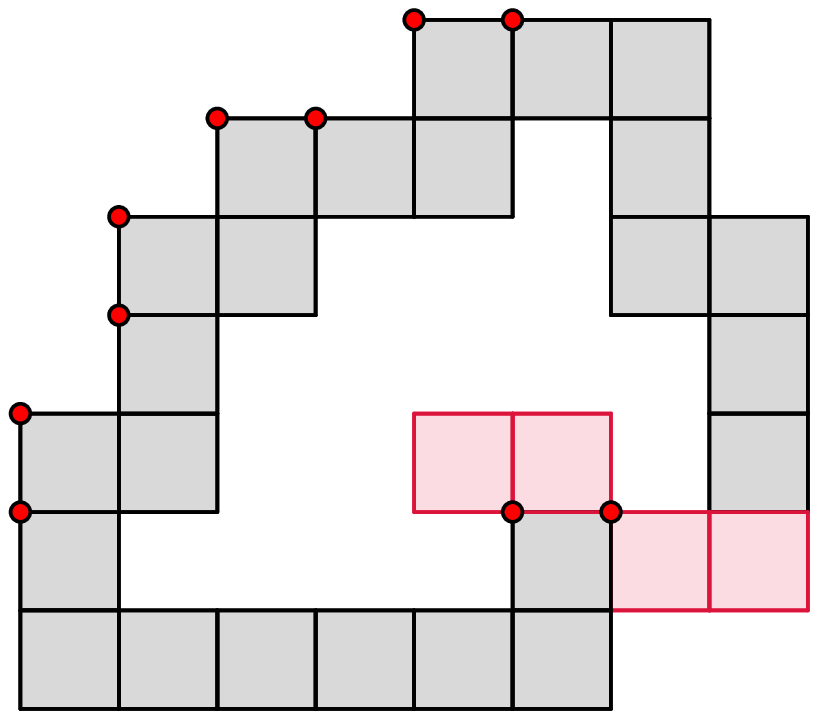}}
\caption{The highlighted points belong to $Y$}
\label{Figure:without_W}
\end{figure}

\begin{rmk}\rm
By the same arguments, the statement of Theorem~\ref{without_W} holds also for $Y=\bigcup_{\cC\in \mathcal{R}}\{a_\cC,b_\cC\}$.
\end{rmk}


\begin{thm}
Let $\cP$ be a closed path polyomino not containing any RW-heptomino. Let $\mathcal{R}_1$ be the set of all LD-horizontal and vertical skew hexominoes contained in $\cP$ and let $\mathcal{R}_2$ be the set of all W-pentominoes contained in $\cP$. Let $Y=(\bigcup_{\mathcal{D}\in \mathcal{R}_1}\{a_D,b_D\})\cup (\bigcup_{\mathcal{W}\in \mathcal{R}_2}\{x_W,y_W\})$. Then $\mathcal{G}$ is the reduced Gr\"obner basis of $I_\cP$ with respect to the monomial order $<^{Y}_\mathrm{lex}$.
\label{without_RW}
\end{thm}

\begin{proof}

Let $f=x_px_q-x_rx_s$ and $g=x_ux_v-x_wx_z$ be the two binomials attached respectively to the inner intervals $[p,q]$ and $[u,v]$ of $\cP$. We discuss the case $|\{p,q,r,s\}\cap \{u,v,w,z\}|=1$, where $[p,q]$ is not contained in $[u,v]$ or vice versa. The cases $q=v$, $r=w$, $p=u$ and $s=z$, as well as $q=w$, $s=u$, $v=r$ and $z=p$, can be proved as in Theorem \ref{without_W}. Suppose $q=u$ and let $A_k$ and $A_{k+2}$ be the cells of $\cP$ having respectively $q$ as upper right and lower left corner. If $\{A_{k-1},A_{k},A_{k+1},A_{k+2},A_{k+3} \}$ is an $L$-configuration the claim follows either if $\gcd(\lt(f),\lt(g))= 1$ or by applying Lemma \ref{intervalli8} if $\gcd(\lt(f),\lt(g))\neq 1$. If $\{A_{k-2},A_{k-1},A_k,A_{k+1},A_{k+2},A_{k+3} \}$ or $\{A_{k-1},A_k,A_{k+1},A_{k+2},A_{k+3},A_{k+4} \}$ is an LD-horizontal or vertical skew hexomino, then $\gcd(\lt(f),\lt(g))=1$. Since there does not exist any RW-heptomino, the last possibilities consist in being $A_{k-1}$, $A_k$, $A_{k+1}$ or $A_{k+2}$ the middle cell of a $W$-pentomino. In all these cases we have the desired conclusion either if $\gcd(\lt(f),\lt(g))= 1$ or by applying Lemma \ref{intervalli8} if $\gcd(\lt(f),\lt(g))\neq 1$. The cases $s=w$, $z=r$ or $v=p$ can be proved by similar arguments.
\end{proof}

\begin{rmk}\rm
With the same arguments, the statement of Theorem~\ref{without_RW} holds also considering $Y=(\bigcup_{\mathcal{D}\in \mathcal{R}_1}\{a_D,b_D\})\cup (\bigcup_{\mathcal{W}\in \mathcal{R}_2}\{x_W,z_W\})$.
\end{rmk}

\noindent Given a closed path polyomino $\cP$ containing both W-pentominoes and RW-heptominoes, our aim is to find a $\mathsf{P}$-order $<^Y$, for a suitable set $Y\subset V(\cP)$, such that $\mathcal{G}$ is the Gr\"obner basis of $I_\cP$ with respect to the monomial order $<^{Y}_{\mathrm{lex}}$. We are going to define the set $Y$ by combining the previous construction and the highlighted points in Figures~\ref{Figura:pentomino}, \ref{Figura:skew-hexomino}, \ref{Figura:heptomino}, and proceeding iteratively from the structure of the polyomino and the arrangement of the cells. In order to simplify notations and writings, we summarize in the table in Figure~\ref{table} the arrangements with highlighted points already introduced in the previous definitions that are useful to define the new set $Y$. We build up the set $Y$ using the algorithm explained below, for which it is also important to consider the configurations described in Figure~\ref{table1IF} and Figure~\ref{table2IF}. 

\begin{figure}[h!]
	\begin{tabular}{|c|c|c|c|c|}
		\hline & I & II & III & IV \\
		\hline
		A & \includegraphics[scale=0.65]{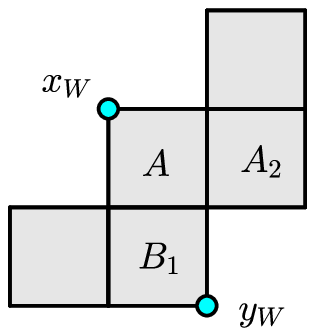} & \includegraphics[scale=0.65]{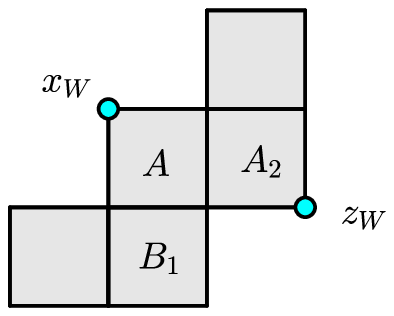} & \includegraphics[scale=0.65]{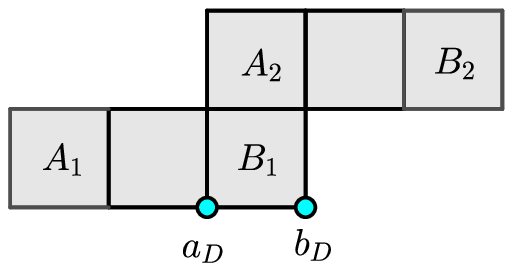}  & \includegraphics[scale=0.65]{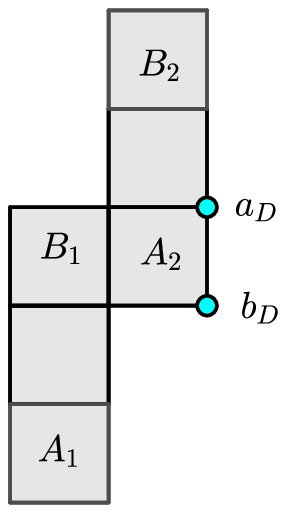} \\ \hline
		B & \includegraphics[scale=0.65]{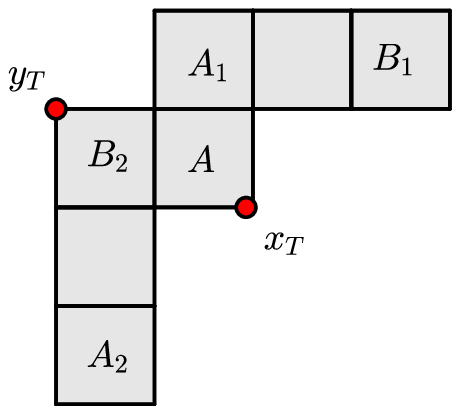} & \includegraphics[scale=0.65]{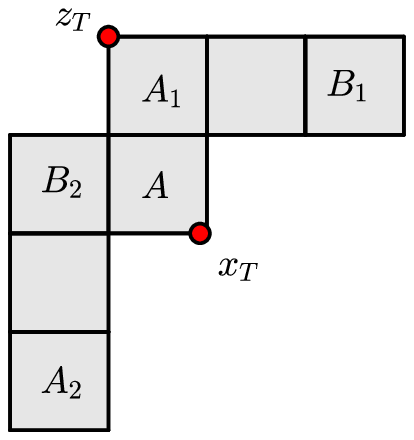} & \includegraphics[scale=0.65]{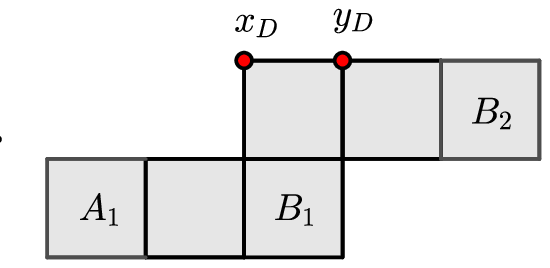} & \includegraphics[scale=0.65]{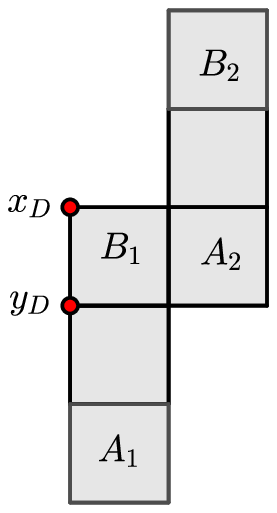}\\
		\hline
	\end{tabular}
	\caption{}
	\label{table}
\end{figure}

\begin{alg} \rm
	Let $\cP$ be a closed path polyomino, whose sequence of cells is $A_1,A_2,\ldots,A_n,A_{n+1}$ (with $A_1=A_{n+1}$) and containing both W-pentominoes and RW-heptominoes. Let $i,j\in \{1,2,\ldots,n,n+1\}$ with $i<j$. We define $Y_{i,j}\subset V(\cP)$ be the set provided by the algorithmic scheme described below:
	
	\begin{enumerate}
		\item Start with $Y_{i,j}=\emptyset$.
		\item Define $\mathcal{Q}=\{q\in \{i,\ldots,j\}\mid A_{q}\ \mbox{is the middle cell of a RW-heptomino}\}$.
		\item If $\mathcal{Q}\neq \emptyset$ define $q_1=\min \mathcal{Q}$, otherwise define $q_1=j$.
		\item FOR $k\in\{i,\ldots,q_1\}$ DO:
		\begin{enumerate}
			\item  IF $A_k$ is the middle cell of a W-pentomino THEN $Y_{i,j}=Y_{i,j}\cup \{x_W,z_W\}$ with reference to II-A of Figure~\ref{table}.
			\item IF $A_k,A_{k+1},\ldots,A_{k+6}$ is a sequence of cells of an LD-horizontal skew hexomino THEN $Y_{i,j}=Y_{i,j}\cup \{a_D,b_D\}$ with reference to III-A of Figure~\ref{table}.
			\item IF $A_k,A_{k+1},\ldots,A_{k+6}$ is a sequence of cells of an LD-vertical skew hexomino THEN $Y_{i,j}=Y_{i,j}\cup \{a_D,b_D\}$ with reference to IV-A of Figure~\ref{table}.
		\end{enumerate}
		\item Define $\mathcal{R}=\{r\in \{q_1+1,\ldots,j\}\mid A_{r}\ \mbox{is the middle cell of a W-pentomino}\}$.
		\item If $\mathcal{R}\neq \emptyset$ define $r_1=\min \mathcal{R}$, otherwise define $r_1=j$.
		\item Define $Q=q_1$ and $R=r_1$.
		\item Consider the RW-heptomino with middle cell $A_Q$ and let $M=\max\{m\in \{i,\ldots,Q\}\mid A_m\cap Y_{i,j}\neq \emptyset\}$.
		\item FOR $k\in\{Q,\ldots,R\}$ DO:
		\begin{figure}[h]
			\begin{tabular}{|c|c|}
				\hline
				\includegraphics[scale=0.65]{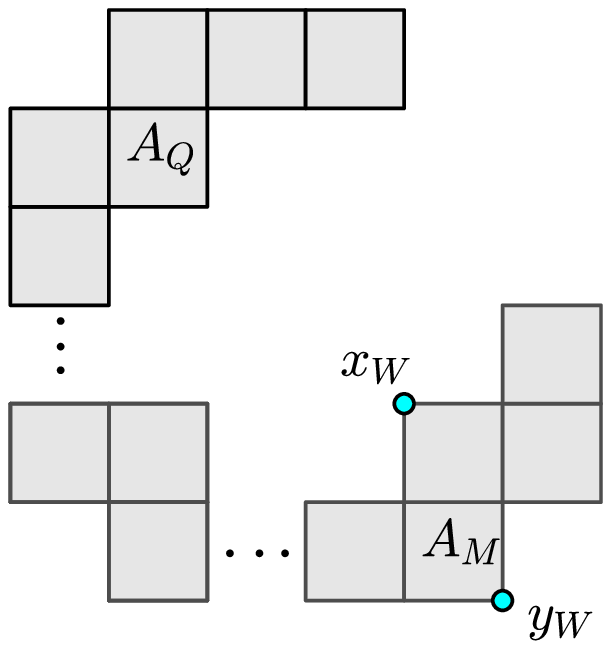} & \includegraphics[scale=0.65]{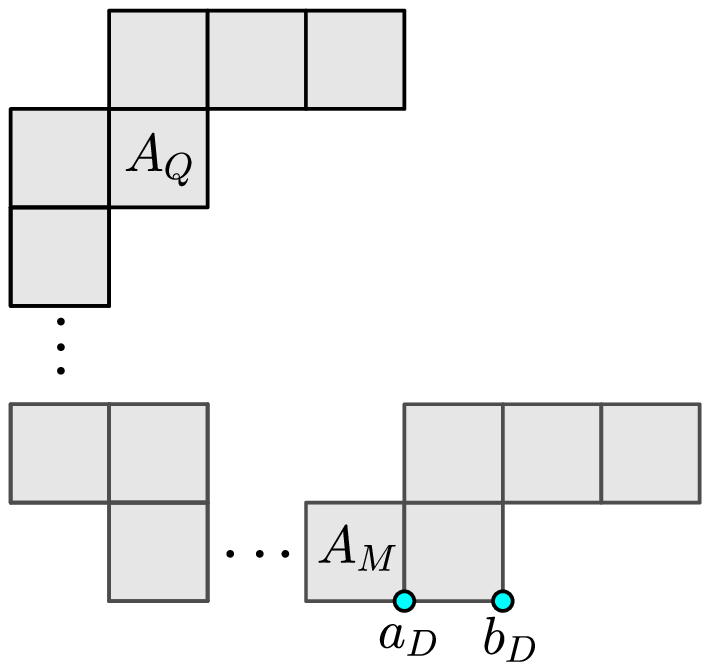} \\ \hline \includegraphics[scale=0.65]{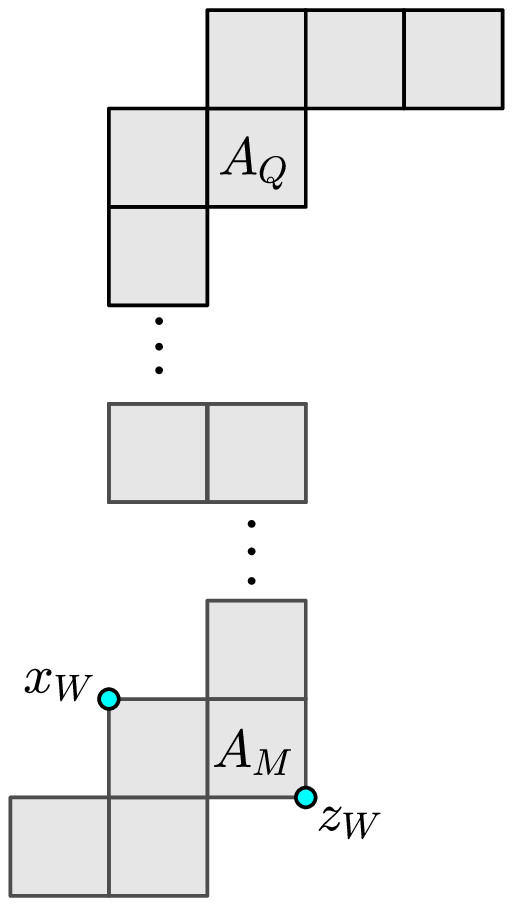}  & \includegraphics[scale=0.65]{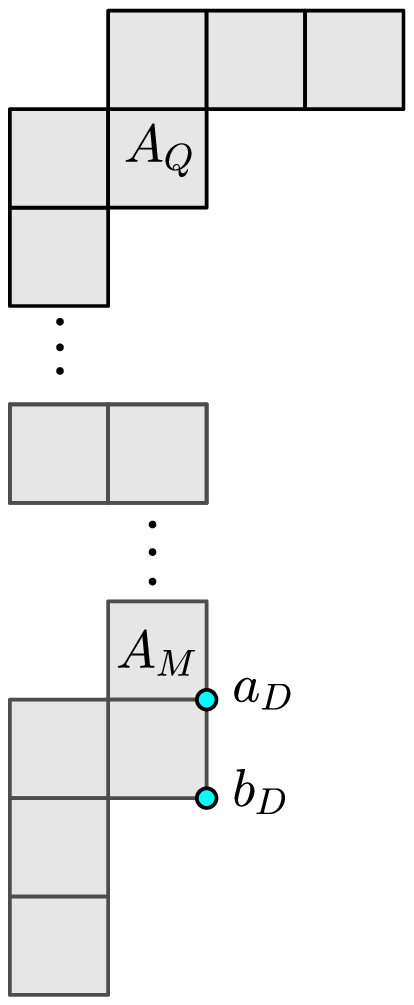} \\ 
				\hline
			\end{tabular}
			\caption{Conflicting configurations with I-B}
			\label{table1IF}
		\end{figure}
		\begin{enumerate}
				\item IF $A_k$ is the middle cell of an RW-heptomino THEN 
			\begin{itemize}
				\item[IF] $A_M$ and $A_{Q}$ do not occur as in the configurations of Figure~\ref{table1IF} 
				\item[THEN] $Y_{i,j}=Y_{i,j}\cup \{x_T,y_T\}$ with reference to I-B of Figure~\ref{table}
				\item[ELSE] $Y_{i,j}=Y_{i,j}\cup \{x_T,z_T\}$ with reference to II-B of Figure~\ref{table}. 
			\end{itemize}
			\item IF $A_k,A_{k+1},\ldots,A_{k+6}$ is a sequence of cells of an LD-horizontal skew hexomino THEN $Y_{i,j}=Y_{i,j}\cup \{x_D,y_D\}$ with reference to III-B of Figure~\ref{table}.
			\item IF $A_k,A_{k+1},\ldots,A_{k+6}$ is a sequence of cells of an LD-vertical skew hexomino THEN $Y_{i,j}=Y_{i,j}\cup \{x_D,y_D\}$  with reference to IV-B of Figure~\ref{table}.
			\item IF $R=j$ THEN RETURN $Y_{i,j}$.
		\end{enumerate}
	
		\item Define $\mathcal{Q}=\{q\in \{r_1+1,\ldots,j\}\mid A_{q}\ \mbox{is the middle cell of a RW-heptomino}\}$.
		\item If $\mathcal{Q}\neq \emptyset$ define $q_2=\min \mathcal{Q}$, otherwise define $q_2=j$.
		\item Define $R=r_1$ and $Q=q_2$.
		\item Consider the W-pentomino with middle cell $A_R$ and let $M=\max\{m\in \{i,\ldots,R\}\mid A_m\cap Y_{i,j}\neq \emptyset\}$.
		\item  FOR $k\in\{R,\ldots,Q\}$ DO:
		\begin{enumerate}
			\item IF $A_k$ is the middle cell of a W-pentomino THEN
			\begin{itemize}
				\item[IF] $A_M$ and $A_{Q}$ do not occur as in the configurations of Figure~\ref{table2IF} 
				\item[THEN] $Y_{i,j}=Y_{i,j}\cup \{x_W,y_W\}$ with reference to I-A of Figure~\ref{table}
				\item[ELSE] $Y_{i,j}=Y_{i,j}\cup \{x_W,z_W\}$ with reference to II-A of Figure~\ref{table}.
			\end{itemize}
			\item IF $A_k,A_{k+1},\ldots,A_{k+6}$ is a sequence of cells of an LD-horizontal skew hexomino THEN $Y_{i,j}=Y_{i,j}\cup \{a_D,b_D\}$ with reference to III-A of Figure~\ref{table}.
			\item IF $A_k,A_{k+1},\ldots,A_{k+6}$ is a sequence of cells of an LD-vertical skew hexomino THEN $Y_{i,j}=Y_{i,j}\cup \{a_D,b_D\}$ with reference to IV-A of Figure~\ref{table}.
			\item IF $Q=j$ THEN RETURN $Y_{i,j}$.
		\end{enumerate}
	
	\begin{figure}[h]
		\begin{tabular}{|c|c|}
			\hline
			\includegraphics[scale=0.65]{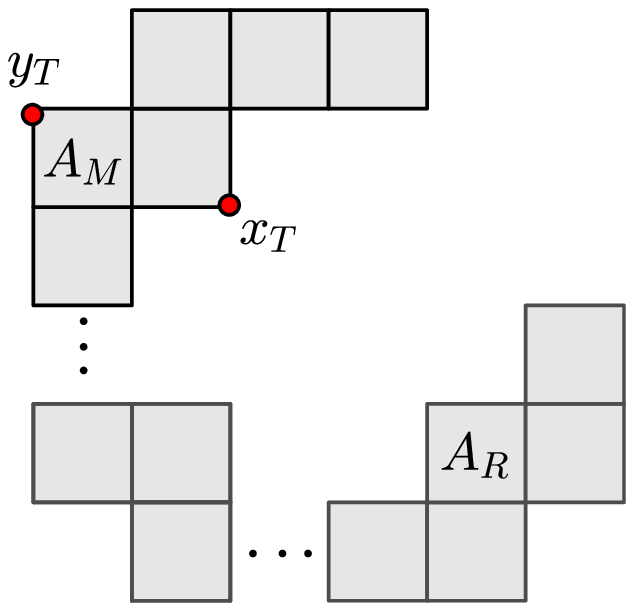} & \includegraphics[scale=0.65]{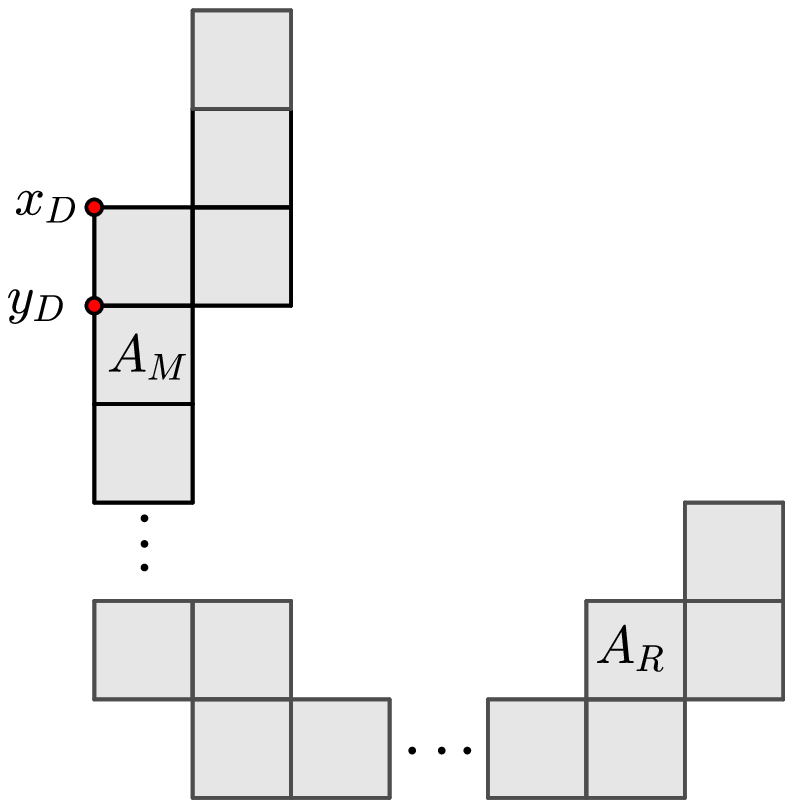} \\ \hline \includegraphics[scale=0.65]{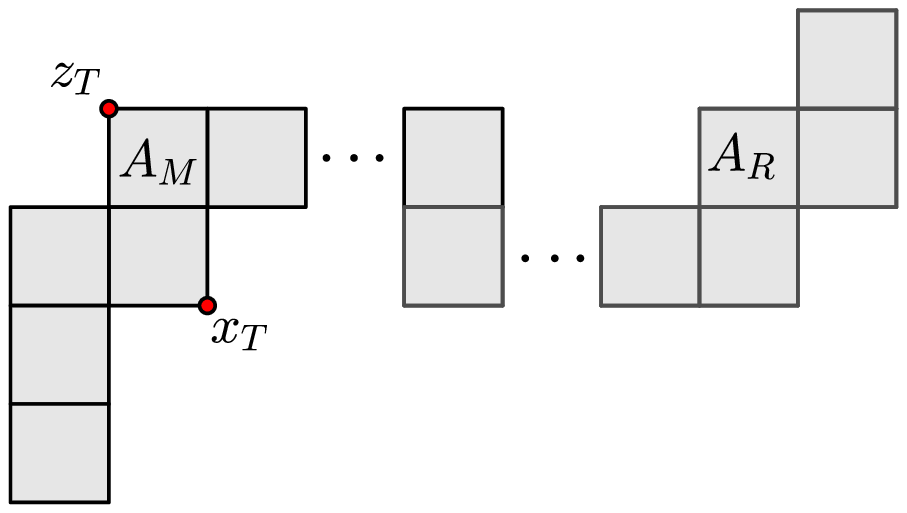}  & \includegraphics[scale=0.65]{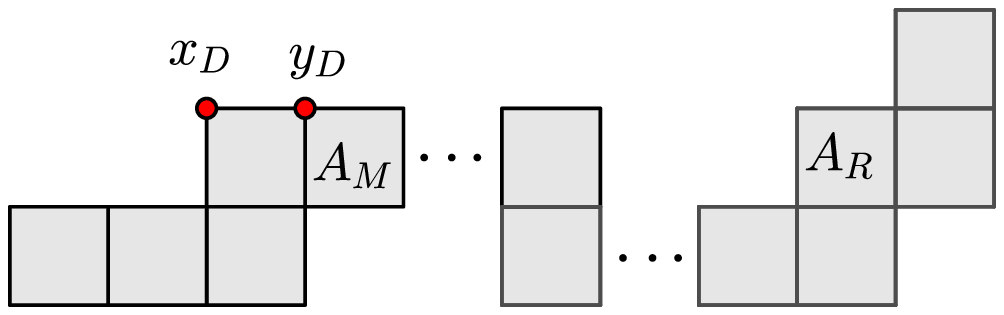} \\ 
			\hline
		\end{tabular}
		\caption{Conflicting configurations with I-A}
		\label{table2IF}
	\end{figure}

		\item $\ell=2$.
		\item WHILE $\ell >1$ DO
		\begin{enumerate}
			\item Define $\mathcal{R}=\{r\in \{q_\ell+1,\ldots,j\}\mid A_{r}\ \mbox{is the middle cell of a W-pentomino}\}$.
			\item If $\mathcal{R}\neq \emptyset$ define $r_\ell=\min \mathcal{R}$, otherwise define $r_\ell=j$.
			\item Define $Q=q_\ell$ and $R=r_\ell$.
			\item $M=\max\{m\in \{i,\ldots,Q\}\mid A_m\cap Y_{i,j}\neq \emptyset\}$.
			\item Execute the instructions in (9).
			\item Define $\mathcal{Q}=\{q\in \{r_\ell+1,\ldots,j\}\mid A_{q}\ \mbox{is the middle cell of a RW-heptomino}\}$.
			\item If $\mathcal{Q}\neq \emptyset$ define $q_{\ell+1}=\min \mathcal{Q}$, otherwise define $q_{\ell+1}=j$.
			\item Define $R=r_\ell$ and $Q=q_{\ell+1}$.
			\item $M=\max\{m\in \{i,\ldots,R\}\mid A_m\cap Y_{i,j}\neq \emptyset\}$.
			\item Execute the instructions in in (14).
			\item $\ell=\ell+1$.
		\end{enumerate}
		\item END
	\end{enumerate}
	\noindent Observe that, since $r_\ell<r_{\ell+1}$ and $q_\ell<q_{\ell+1}$ for all $\ell \in \mathbb{N}$ then there exists $\overline{\ell}$ such that $r_{\overline{\ell}}=j$ or $q_{\overline{\ell}}=j$, so the procedure stops and the set $Y_{i,j}$ is returned.  
	\label{def:Yij}
\end{alg}

\begin{defn}\rm 
Let $\cP$ be a closed path polyomino containing both W-pentominoes and RW-heptominoes. Consider a $W$-pentomino $\cW$ of $\cP$ and suppose that $\cW$ contains the cells $A_1$, $A_2$, $A_3$, $A_4$ and $A_5$, labelled bottom up as in Figure \ref{Figura:W pentomino definizione Yij}. We put $L=Y_{2,n+1}$.  In Figure \ref{Figure:Esempio costruzione Y_ij} we make in evidence, for instance, the points belonging to $L$.
\begin{figure}[h]
\includegraphics[scale=0.7]{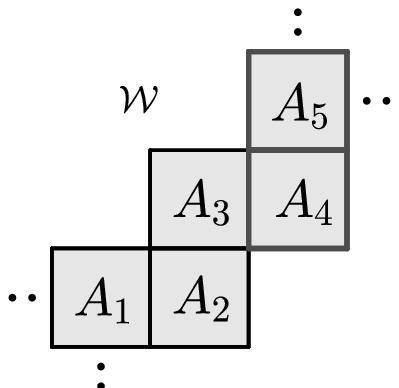}
\caption{}
\label{Figura:W pentomino definizione Yij}
\end{figure}
\label{def:W-pentomino}
\end{defn}


\begin{figure}[h]
	\centering
	\includegraphics[scale=0.55]{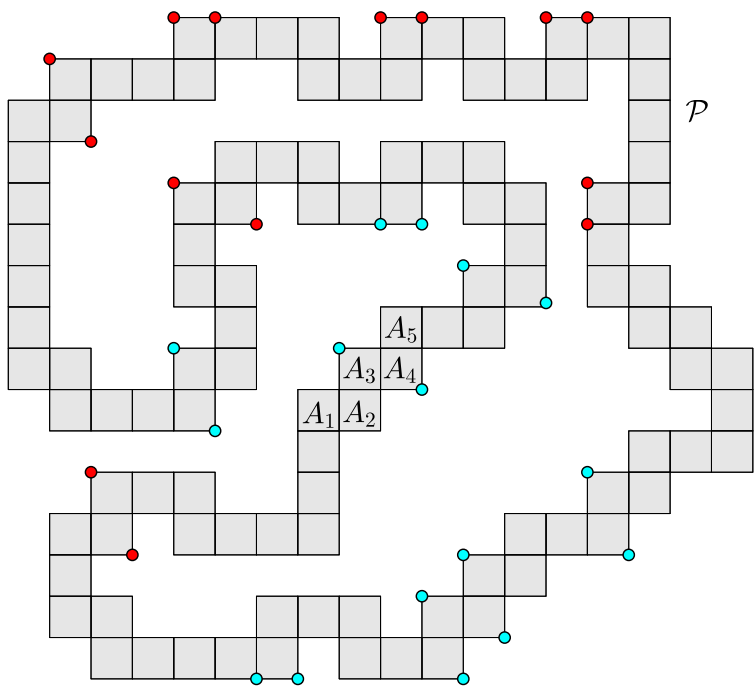}
	\caption{The set $Y_{2,n+1}\subset V(\cP)$ consists of the highlighted points}
	\label{Figure:Esempio costruzione Y_ij}
\end{figure}

\begin{thm}
Let $\cP$ be a closed path. Suppose that $\cP$ contains a W-pentomino and an RW-heptomino and let $L$ be the set given in Definition~\ref{def:W-pentomino}. Then $\mathcal{G}$ is the reduced Gr\"obner basis of $I_\cP$ with respect to $<^{L}_\mathrm{lex}$.
\label{teorema: W_pentomino + RW}
\end{thm}

\begin{proof}
 Let $f$ and $g$ be the two binomials attached respectively to the inner intervals $[p,q]$ and $[u,v]$ of $\cP$. It suffices to show that $S(f,g)$ reduces to $0$ modulo $\cG$ in every case. Observe that the desired claim follows from Definitions \ref{def:Yij} and \ref{def:W-pentomino}, arguing as in Theorem \ref{without_RW}. In fact,  we always have that either $\gcd(\lt(f),\lt(g))= 1$ or, if $\gcd(\lt(f),\lt(g))\neq 1$, it is sufficient to apply the lemmas of Section \ref{Section: Properties}.
\end{proof}

\begin{thm}
Let $\cP$ be a closed path polyomino having an $L$-configuration or a ladder of at least three steps, or equivalently having no zig-zag walks. Then $K[\cP]$ is a normal Cohen-Macaulay domain.
\end{thm}

\begin{proof}
From Theorem~\ref{teorema: W_pentomino + RW} we obtain that there exists a monomial order $\prec$ such that $\mathcal{G}$ is the Gr\"obner basis of $I_\cP$ with respect to $\prec$, in particular $I_\cP$ admits a squarefree initial ideal  with respect to some monomial order. Since $\cP$ has an $L$-configuration or a ladder of three steps, from \cite{Cisto_Navarra} we have that $I_{\cP}$ is a toric ideal. By \cite[Corollary 4.26]{binomial ideals} we obtain that $K[\cP]$ is normal and by \cite[Theorem 6.3.5]{Bruns_Herzog} we obtain that $K[\cP]$ is Cohen-Macaulay.
\end{proof}

\begin{rmk}\rm
	In \cite{def balanced} the authors proved that if $\cP$ is a balanced polyomino, equivalently $\cP$ is simple, then the universal Gr\"obner basis is squarefree. In general this fact does not hold for a non-simple polyomino. Consider the closed path $\cP$ in Figure \ref{Immagine: esempio polimino Graver}.
	\begin{figure}[h]
		\centering
		\includegraphics[scale=1]{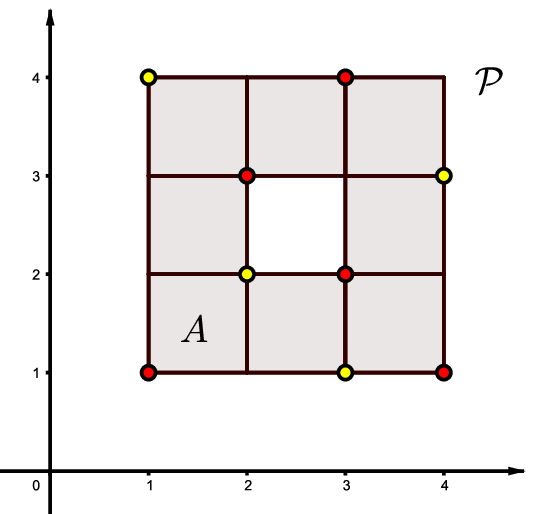}
		\caption{}
		\label{Immagine: esempio polimino Graver}
		\end{figure}
	Let $\{V_1,V_2,V_3,V_4\}$ and $\{H_1,H_2,H_3,H_4\}$ be respectively the sets of the maximal vertical and horizontal edge intervals of $\cP$ such that $r=(i,j)\in V_i\cap H_j$, and let $\{v_1,v_2,v_3,v_4\}$ and $\{h_1,h_2,h_3,h_4\}$ be the associated sets of the variables. Let $w$ be another variable different from $v_i$ and $h_j$. We recall from \cite{Cisto_Navarra} that $I_{\cP}=J_{\cP}$, where $J_{\cP}$ is the kernel of $\phi$, defined as 
	\begin{align*}
	\phi: K[x_{ij}:(i,j)\in V(\cP)] &\longrightarrow K[\{v_i,h_j,w\}:i,j \in \{1,2,3,4\}]\\
	\phi(x_{ij}&)=v_ih_jw^k
	\end{align*}
    where $k=0$ if $(i,j)\notin A$, and $k=1$, if $(i,j)\in A$.\\
  Consider the binomial $f=x_{11}x_{23}x_{32}x_{34}x_{41}-x_{14}x_{22}x_{31}^2x_{43}$ attached to the vertices in red and yellow. Observe that $f\in I_\cP$ because $\phi(x_{11}x_{23}x_{32}x_{34}x_{41})=\phi(x_{14}x_{22}x_{31}^2x_{43})$. We show that $f$ is primitive, that is there does not exist any binomial $g=g^+-g^-$ in $I_{\cP}$ with $g\neq f$ such that $g^+|x_{11}x_{23}x_{32}x_{34}x_{41}$ and $g^-|x_{14}x_{22}x_{31}^2x_{43}$. Suppose by contradiction that there exists such a binomial. 
Observe that $2<\deg(g)< 5$, since $f\neq g$ and all binomials of degree two satisfying the primitive conditions are not inner $2$-minors. It is sufficient to prove that $x_{11}$ (resp. $x_{22}$) cannot divide $g^+$ (resp. $g^-$). If that happens, then $w$ divides $\phi(g^+)$, which is equal to $\phi(g^-)$, so $x_{22}$ divides $g^-$. 
  Since $g\in I_{\cP}=J_{\cP}$, in particular $\phi(g^+)=\phi(g^-)$, we obtain that $g^+=x_{11}x_{23}x_{32}x_{34}x_{41}$ and $g^-=x_{14}x_{22}x_{31}^2x_{43}$ from easy calculations. Hence $f=g$, a contradiction. In conclusion we have that $f$ is a primitive binomial of $I_{\cP}$. Since for a toric ideal the universal Gr\"obner basis coincides with the Graver basis (see \cite{Graver=Universal for toric}), the primitive binomials of $I_{\cP}$ form the universal Gr\"obner basis $\mathscr{G}$ of $I_{\cP}$. Since $f$ is a primitive binomial of $I_{\cP}$, it follows that $\mathscr{G}$ is not squarefree. Anyway $I_{\cP}$ is a radical ideal which admits a squarefree initial ideal with a different monomial ordering, for instance with respect to  $<^1_{\mathrm{lex}}$, since the set of generator of $I_{\cP}$ is the reduced Gr\"obner basis by \cite[Theorem 4.1]{Qureshi}.
\end{rmk}

\noindent We conclude providing some questions which follow immediately from the results of this paper.
\begin{itemize}
\item We ask if the initial ideal of $I_{\cP}$, attached to a (weakly) closed path, with respect to the monomial orders $<^1_{\mathrm{lex}}$ and $<^2_{\mathrm{lex}}$ defined in \cite{Qureshi} is squarefree.
\item With reference to \cite{Cisto_Navarra2}, we ask if also for all weakly closed path polyominoes there exist some monomial orders such that the set of the generators of the ideal is the reduced Gr\"obner basis. 
\end{itemize}

\subsection*{Acknowledgements}
We wish to thank Ayesha Asloob Qureshi for comments on an earlier version of this paper.

	\end{document}